\documentclass[11pt]{article}
\usepackage[top=2.54cm, bottom=2.54cm, left=2.8cm, right=2.8cm]{geometry}
\usepackage[tbtags]{amsmath}
\usepackage{amssymb}
\usepackage{amsthm}
\usepackage{mathtools}
\usepackage{bm}
\usepackage{fancyhdr}
\usepackage{latexsym}
\usepackage{mathrsfs}
\usepackage{wasysym}
\usepackage{fancyhdr}
\usepackage{float}
\usepackage{graphicx}
\usepackage[numbers,sort&compress]{natbib}
\usepackage{setspace}
\allowdisplaybreaks[4]

\usepackage[dvipsnames]{xcolor}

\newtheorem{theorem}{\bf Theorem}[section]
\newtheorem{lemma}[theorem]{Lemma}
\newtheorem{proposition}[theorem]{Proposition}

\newtheorem{corollary}[theorem]{Corollary}
\newtheorem{conj}[theorem]{Conjecture}
\newtheorem{claim}{\indent Claim}[]
\newtheorem*{obser}{Observation}

\UseRawInputEncoding 

\begin{document}
	\begin{spacing}{1.1}
		\title{
			Maximum size of $C_{\leq k}$-free strong digraphs with out-degree at least two}
		\author{Bin Chen$^a$, \, Xinmin Hou$^{b,c}$\\
			\small$^a$Hefei National Laboratory\\
			\small University of Science and Technology of China, Hefei 230088, China\\
			\small$^b$ School of Mathematics Sciences\\
			\small University of Science and Technology of China, Hefei, Anhui 230026, China\\
			\small $^{c}$CAS Key Laboratory of Wu Wen-Tsun Mathematics\\
			\small University of Science and Technology of China, Hefei, Anhui 230026, China\\
			\small Email: $^a$cb03@ustc.edu.cn, \,\, $^b$xmhou@ustc.edu.cn }
		\date{}
		\maketitle
		{\bf Abstract}\,:
		Let $\mathscr{H}$ be a family of digraphs. A digraph $D$ is \emph{$\mathscr{H}$-free} if it contains no isomorphic copy of any member of $\mathscr{H}$. For $k\geq2$, we set $C_{\leq k}=\{C_{2}, C_{3},\ldots,C_{k}\}$, where $C_{\ell}$ is a directed cycle of length $\ell\in\{2,3,\ldots,k\}$. Let $D_{n}^{k}(\xi,\zeta)$ denote the family of \emph{${C}_{\le k}$-free} strong digraphs on $n$ vertices with every vertex having out-degree at least $\xi$ and in-degree at least $\zeta$, where both $\xi$ and $\zeta$ are positive integers. Let $\varphi_{n}^{k}(\xi,\zeta)=\max\{|A(D)|:\;D\in D_{n}^{k}(\xi,\zeta)\}$ and $\Phi_{n}^{k}(\xi,\zeta)=\{D\in D_{n}^{k}(\xi,\zeta): |A(D)|=\varphi_{n}^{k}(\xi,\zeta)\}$. 
		Bermond et al.\;(1980) verified that $\varphi_{n}^{k}(1,1)=\binom{n-k+2}{2}+k-2$. Chen and Chang\;(2021) showed that $\binom{n-1}{2}-2\leq\varphi_{n}^{3}(2,1)\leq\binom{n-1}{2}$. This upper bound was further improved to  $\binom{n-1}{2}-1$ by Chen and Chang\;(DAM, 2022), furthermore, they also gave the exact values of $\varphi_{n}^{3}(2,1)$ for $n\in \{7,8,9\}$. In this paper, we continue to determine the exact values of $\varphi_{n}^{3}(2,1)$ for $n\ge 10$, i.e., $\varphi_{n}^{3}(2,1)=\binom{n-1}{2}-2$ for $n\geq10$.
		
		{\bf AMS}\,: 05C20; 05C35; 05C38.
		
		
		\section{Introduction}
		
		\noindent
		
		For terminology not explicitly introduced here, we refer to \cite{1}. All digraphs considered here are finite and simple. Let $D=(V, A)$ be a digraph with vertex set $V=V(D)$ and arc set $A=A(D)$. The \emph{order}\;(\emph{size}) of $D$ is the number of vertices\;(arcs) in $D$.

		Given a family of digraphs $\mathscr{H}$, a digraph $D$ is called \emph{$\mathscr{H}$-free} if it has no isomorphic copy of any member of $\mathscr{H}$. The \emph{Tur\'{a}n number} of $\mathscr{H}$, denoted $ex(n,\mathscr{H})$, is defined to be the maximum size of $\mathscr{H}$-free digraphs on $n$ vertices.

		The study of \emph{Tur\'{a}n-type problems} is one of the central topics in extremal graph theory and it can be traced back to the work of Tur\'{a}n \cite{13} in 1941, known as Tur\'{a}n's Theorem, which is a generalization of Mantel's Theorem due to Mantel \cite{14} in 1907. Since then, Tur\'{a}n-type problems have received considerable attention and have developed extremely rapidly, and a large number of classical and meaningful results for undirected graphs have been given by researchers. Brown and Harary \cite{5} were the first to investigate such problems for digraphs, see, e.g., \cite{5,6,7} for details. However, only a few results have been obtained for digraphs.

		A \emph{directed path} of length $\ell$ of $D$ is a list $x_{0},x_{1},\ldots,x_{\ell}$ of distinct vertices such that $(x_{i},x_{i+1})\in A(D)$ for $0\leq i\leq \ell-1$. A digraph $D$ is \emph{strongly connected}\;(\emph{strong} for short) if every two vertices are mutually reachable by directed paths. A \emph{directed cycle} of length $\ell$ in $D$ is a list $y_{0},y_{1},\ldots,y_{\ell-1}$ of distinct vertices such that $(y_{i},y_{i+1})\in A(D)$ for $0\leq i\leq \ell-2$ and $(y_{\ell-1},y_{0})\in A(D)$. The \emph{girth} $g(D)$ of $D$ is the length of the shortest cycles in $D$. All \emph{paths}\;(resp., \emph{cycles}) considered in this paper are directed paths\;(resp., directed cycles).
		
		Thomassen raised an interesting problem as follows: What is the least number $m(n, k)$ such that every strong digraph with $n$ vertices and size at least $m(n, k)$ has a cycle of length at most $k$\;(see \cite{1}, Chapter $8$, p.331 for details)?
		In 1980, Bermond et al. \cite{2} resolved the above problem by showing the following theorem.
		
		\begin{theorem}[\textnormal{\cite{2}}] \label{k} For $n\geq k+1$ and $k\geq2$,
			$$m(n,k)=\frac{n^{2}+(3-2k)n+k^{2}-k}{2}.$$
			Moreover, a family of extremal digraphs with size $m(n,k)-1$ containing no cycle of length at most $k$ has been recursively constructed.
			
		\end{theorem}

		Let $C_{\leq k}=\{C_{2}, C_{3},\ldots,C_{k}\}$, where $k\geq2$ and $C_{i}$ is a cycle of length $i$ for $i\in \{2,3,\ldots,k\}$. By $D_{n}^{k}(\xi,\zeta)$ we mean a family of \emph{$C_{\leq k}$-free} strong digraphs on $n$ vertices such that every digraph $D\in D_{n}^{k}(\xi,\zeta)$ satisfies $\delta^{+}(D)\geq\xi$ and $\delta^{-}(D)\geq\zeta$, where both $\xi$ and $\zeta$ are positive integers. Let 
		$$\varphi_{n}^{k}(\xi,\zeta)=\max\{|A(D)|:D\in D_{n}^{k}(\xi,\zeta)\},$$ 
		which can be viewed as the Tur\'{a}n number of \emph{$C_{\leq k}$} in the host family of  strong digraphs on $n$ vertices with  $\delta^{+}(D)\geq\xi$ and $\delta^{-}(D)\geq\zeta$. 
		Denote by $$\Phi_{n}^{k}(\xi,\zeta)=\{D\in D_{n}^{k}(\xi,\zeta): |A(D)|=\varphi_{n}^{k}(\xi,\zeta)\}$$ and call graphs $D\in\Phi_{n}^{k}(\xi,\zeta)$ the extremal graphs of $C_{\le k}$. Note that for a given collection $\{n,k,\xi,\zeta\}$, it is possible that $D_{n}^{k}(\xi,\zeta)=\O$ and we define $\varphi_{n}^{k}(\xi,\zeta)=0$ if it occurs. Clearly, $\varphi_{n}^{k}(\xi,\zeta)=\varphi_{n}^{k}(\zeta,\xi)$ for all collections of $\{n,k,\xi,\zeta\}$. 
		Theorem~\ref{k} tells us that $$\varphi_{n}^{k}(1,1)=m(n,k)-1=\frac{n^{2}+(3-2k)n+k^{2}-k-2}{2}$$ for $n\geq k+1$ and $k\geq2$. For $k=2$, it is easy to see that $\Phi_{n}^{2}(1,1)=\{T_n\}$ for $n\geq3$, where $T_n$ is a strong tournament on $n$ vertices,  and $\Phi_{n}^{2}(1,1)=\O$ for $n\leq 2$.

		One motivation of the study of the function $\varphi_{n}^{k}(\xi,\zeta)$ comes from the following two well-known conjectures proposed by Behzad et al. \cite{26} in 1970
		and Caccetta and H\"{a}ggkvist \cite{4} in 1978.
		\begin{conj}[Behzad, Chartrand and Wall, 1970, \textnormal{\cite{26}}]\label{r}
			Every $r$-regular digraph on $n$ vertices has a cycle of length at most $\lceil n/r\rceil$.
		\end{conj}

		\begin{conj}[Caccetta and H\"{a}ggkvist, 1978, \textnormal{\cite{4}}]\label{minimum out-degree}
			Every digraph on $n$ vertices of minimum out-degree at least $r$ has a cycle of length at most $\lceil n/r\rceil$.
		\end{conj}
		
		The two conjectures remain widely open. Many results on these two conjectures have been obtained, see, e.g., \cite{16,18,19,20,21,22,23,30,31,32,33,34,35,36,37}.
		The following is a simple observation. 
		\begin{obser}\label{equivalence}
			(1) Behzad-Chartrand-Wall Conjecture is true if and only if $\varphi_{n}^{\lceil n/r\rceil}(r,r)=0$, i.e., $D_{n}^{\lceil n/r\rceil}(r,r)=\O$.
			
			(2)		Caccetta-H\"{a}ggkvist Conjecture is true if and only if $\varphi_{n}^{\lceil n/r\rceil}(r,1)=0$, i.e., $D_{n}^{\lceil n/r\rceil}(r,1)=\O$.
		\end{obser}
		
		The above observation implies that determining the exact values of $\varphi_{n}^{k}(\xi,\zeta)$ seems quite difficult.
		A special case of Conjecture~\ref{minimum out-degree}, which has received substantial attention, states that $\varphi_{n}^{3}(n/3,1)=0$.
		The best known result for this special case, given by Hladk\'{y} et al.  \cite{23}, is $\varphi_{n}^{3}(0.3465n,1)=0$. 
		As a straightforward corollary of Theorem~\ref{k} for $k=3$, we have
		\begin{corollary}[\textnormal{\cite{2}}]\label{n=3}
			For $n\geq4$, $$\varphi_{n}^{3}(1,1)=\binom{n-1}{2}+1.$$
		\end{corollary}

		Note that $\varphi_{n}^{k}(\xi^{'},\zeta^{'})\leq \varphi_{n}^{k}(\xi,\zeta)$ if $1\leq\xi\leq \xi^{'}$ and $1\leq\zeta\leq \zeta^{'}$. Thereby, $\varphi_{n}^{3}(\xi,1)\leq \varphi_{n}^{3}(1,1)=\binom{n-1}{2}+1$ for every $\xi\geq1$ and $n\geq4$. Moreover, $\varphi_{n}^{3}(1,1)=0$ for $n\leq 3$ and $\varphi_{n}^{3}(2,1)=0$ for $n\leq6$. Recently, Chen and Chang \cite{3} bounded the values of $\varphi_{n}^{3}(2,1)$ by presenting $\binom{n-1}{2}-2\leq\varphi_{n}^{3}(2,1)\leq\binom{n-1}{2}$ for all $n\geq7$. The lower bound is obtained by constructing a family of digraphs that meet the requirements and the upper bound is obtained by using the structure property of $\Phi_{n}^{3}(1,1)$. The same authors \cite{24} improved the above upper bound to $\binom{n-1}{2}-1$ for $n\geq7$ and further gave the exact values of $\varphi_{n}^{3}(2,1)$ for $7\leq n\leq9$.
		
		\begin{theorem}[\textnormal{\cite{24}}]\label{upper-lower bounds}
			For $n\geq 7$, $$\binom{n-1}{2}-2\leq\varphi_{n}^{3}(2,1)\leq\binom{n-1}{2}-1.$$
			Moreover, $\varphi_{n}^{3}(2,1)=\binom{n-1}{2}-1$ when $n=7,8$ and $\varphi_{n}^{3}(2,1)=\binom{n-1}{2}-2$ when $n=9$.
		\end{theorem}
		

		In this paper, we continue the work of Chang and Chen~\cite{24} and determine the exact values of $\varphi_{n}^{3}(2,1)$ for all $n\geq10$.
		
		\begin{theorem}\label{10}
			For $n\geq 10$, $$\varphi_{n}^{3}(2,1)=\binom{n-1}{2}-2.$$
		\end{theorem}
		
		Consequently, by combining Theorem~\ref{upper-lower bounds} and $\varphi_{n}^{3}(2,1)=0$ for $n\leq 6$, we obtain  the exact values of $\varphi_{n}^{3}(2,1)$ for all $n$.

		The rest of this paper is organized as follows. In the next section, we first introduce some notation, and then we present several auxiliary results that will be used in the proof of Theorem~\ref{10}.
		In Section 3, we mainly verify Theorem~\ref{10}. We end in Section 4 with a few concluding remarks.

		\section{Preliminary}
		
		\noindent
		
		In this section, we introduce some notation and present several auxiliary results to verify our main result. 
		
		\subsection{Notation}
		
		\noindent

		The main aim of this subsection is to establish some notation, see e.g., \cite{1}.

		Let $D$ be a digraph. If $(u,v)$ is an arc of $D$, then we say $u$ dominates $v$ or $v$ is dominated by $u$, and $u$\;(resp., $v$) is its \emph{tail}\;(resp., \emph{head}). For convenience, we use $u\rightarrow v$ to mean an arc $(u,v)$. Every two vertices $u$ and $v$ are \emph{adjacent} if $u\rightarrow v$ or $v\rightarrow u$, otherwise they are nonadjacent.
		
		Let $v$ be a vertex of $D$. The \emph{out-neighborhood} of $v$ is the set $N^{+}_{D}(v)=\{u\in V(D):v\rightarrow u\}$, and the \emph{out-degree} of $v$ is $d^{+}_{D}(v)=|N^{+}_{D}(v)|$. The \emph{in-neighborhood} of $v$ is the set $N^{-}_{D}(v)=\{u\in V(D):u\rightarrow v\}$, and the \emph{in-degree} of $v$ is $d^{-}_{D}(v)=|N^{-}_{D}(v)|$. The \emph{neighborhood} of $v$ is the set $N_{D}(v)=\{u\in V(D):u,v\;\text{are} \;\text{adjacent}\}$, and the \emph{degree} of $v$ is $d_{D}(v)=|N_{D}(v)|$. Note that $N_{D}(v)=N^{+}_{D}(v)\cup N^{-}_{D}(v)$. The vertices in $N^{+}_{D}(v), N^{-}_{D}(v)$ and $N_{D}(v)$ are called the \emph{out-neighbors}, \emph{in-neighbors} and \emph{neighbors} of $v$, respectively. By an \emph{$(\alpha,\beta)$-vertex} we mean a vertex with out-degree $\alpha$ and in-degree $\beta$, where both $\alpha$ and $\beta$ are integers.

		Let $\Delta^{+}(D)$ be the \emph{maximum out-degree}, $\Delta^{-}(D)$ be the \emph{maximum in-degree}, and $\Delta(D)$ be the \emph{maximum degree}. Similarly, we denote by $\delta^{+}(D)$ the \emph{minimum out-degree}, $\delta^{-}(D)$ the \emph{minimum in-degree} and $\delta(D)$ the \emph{minimum degree}.

		A digraph $H$ is a \emph{subdigraph} of $D$ if $V(H)\subseteq V(D)$ and $A(H)\subseteq A(D)$. For every such $H$, we denote $N^{+}_{H}(v)=N^{+}_{D}(v)\cap V(H)$\;(resp., $N^{-}_{H}(v)=N^{-}_{D}(v)\cap V(H)$) and $d^{+}_{H}(v)=|N^{+}_{H}(v)|$\;(resp., $d^{-}_{H}(v)=|N^{-}_{H}(v)|$). For any $W\subseteq V(D)$, let $D\langle W\rangle$ be the subdigraph induced by $W$ in $D$, and we denote the digraph $D\langle V(D)\backslash W\rangle$ by $D-W$. If $W$ contains only one vertex say $w$, then we will omit the set brackets and write $D\langle w\rangle$ for $D\langle \{w\}\rangle$. By $D-w$ we mean a subdigraph obtained from $D$ by deleting $w$ and all arcs incident with $w$. Let $X$ and $Y$ be two disjoint vertex sets. The arc set from $X$ to $Y$ will be denoted by $A(X,Y)$, i.e., each arc $(x,y)\in A(X,Y)$ satisfies that $x\in X$ and $y\in Y$.

		A maximal strong subdigraph of $D$ is called a \emph{strong component}\;(\emph{component} for short) of $D$. The components of $D$ have an \emph{acyclic ordering}, namely, they can be labeled $D_{1},D_{2},\ldots,D_{h}$ such that all arcs are from $D_{i}$ to $D_{j}$ for $1\leq i<j\leq h$ (see, e.g., \cite{1} for details).

		\vspace{0.1cm}
		
		\noindent\textbf{Definition 2.1.} For a digraph $D$, $\gamma(D)$ is the number of unordered pairs $\{u,v\}$ of distinct vertices $u,v$ such that they are nonadjacent in $D$. Let $Q$ and $K$ be vertex disjoint subdigraphs of $D$. By $\gamma(Q,K)$ we denote the number of unordered pairs of nonadjacent vertices $\{u,v\}$ in $D$ with $u\in V(Q)$ and $v\in V(K)$. Particularly, if $Q$ contains a unique vertex $w$, then we shall use $\gamma(w,K)$ to mean $\gamma(Q,K)$.

		\subsection{Characterization of $\Phi_{n}^{3}(1,1)$}
		
		\noindent
		
		Note that $D_{n}^{3}(1,1)=\O$ for $n\leq3$. It is clear that $D_n^3(1,1)=\{C_4\}$ when $n=4$. 
		In 2021, Chen and Chang \cite{3} characterized all the extremal digraphs in $\Phi_{n}^{3}(1,1)$, and there are $5$ families of such extremal digraphs for $n\geq 5$, denoted by $\mathscr D_{i}$, where $i\in \{1,2,3,4,5\}$, as shown in Fig.1.

		\begin{figure}[!hbpt]
			\begin{center}
				\includegraphics[scale=0.45]{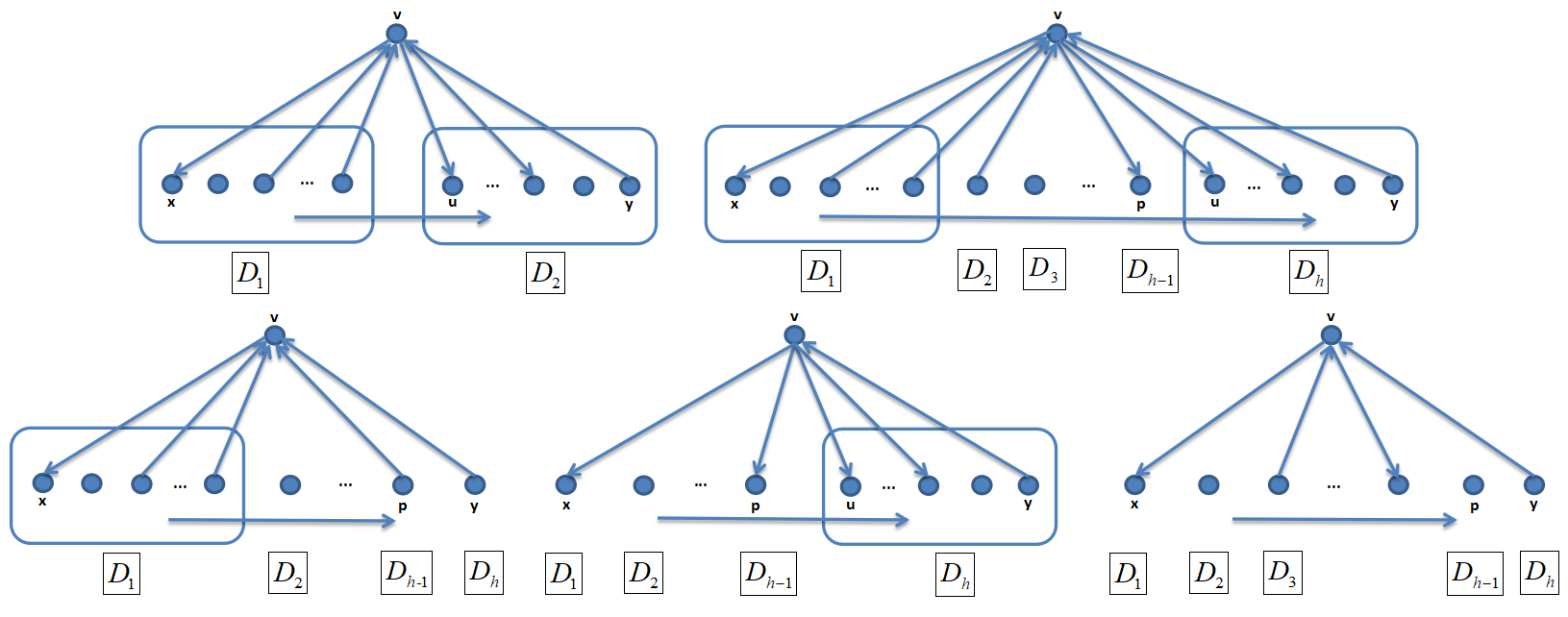}
				\caption{\small
					$\mathscr D_{1},\mathscr D_{2},\mathscr D_{3},\mathscr D_{4}$, and $\mathscr D_{5}$ are arranged in the order of top to bottom, left to right
				}
				\label{fig:framework_of_GP}
			\end{center}
		\end{figure}

		\begin{theorem}[\textnormal{\cite{3}}]\label{5}
			For $n\ge 5$, $\Phi_{n}^{3}(1,1)=\cup_{i=1}^5\mathscr{D}_i$.
		\end{theorem}

		Next, we describe each $\mathscr D_{i}\subseteq \Phi_{n}^{3}(1,1)$ in detail to make this paper self-contained, where $i\in \{1,2,3,4,5\}$ and $n\geq5$.

		Let $D$ denote a digraph of $\Phi_{n}^{3}(1,1)$ for $n\geq 5$. For any fixed $v\in V(D)$ with $d_{D}(v)\leq n-3$, the subdigraph $D-v$ has $h\geq 2$ components, denoted by $D_{1},D_{2},\ldots,D_{h}$, such that they have an acyclic ordering. Additionally, $|V(D_{i})|=n_{i}$ for $i\in\{1,2,\ldots,h\}$.
		
		\vspace{0.2cm}
		
		\noindent\textbf{Characterization I. Orders of $D_{i}$ for $2\leq i\leq h-1$}
		
		\vspace{0.2cm}
		
		\textbf{(1).} If $D\in \mathscr D_{1}$, then $D-v$ contains two components and $n_{i}\geq 4$ for $i=1,2$.

		\textbf{(2).} If $D\in \mathscr D_{i}$ for $i\in \{2,3,4,5\}$, then $h\geq 3$. Each component $D_{i}$ contains a unique vertex, namely, $n_{i}=1$ holds, where $2\leq i\leq h-1$. In addition, we enumerate the orders of $D_{1}$ and $D_{h}$ as below:

		\textbf{(2.1).} If $D$ is a member of $\mathscr D_{2}$, then $n_{1}\geq4$ and $n_{h}\geq4$;
		
		\textbf{(2.2).} If $D$ is a member of $\mathscr D_{3}$, then $n_{1}\geq4$ and $n_{h}=1$;
		
		\textbf{(2.3).} If $D$ is a member of $\mathscr D_{4}$, then $n_{1}=1$ and $n_{h}\geq4$;
		
		\textbf{(2.4).} If $D$ is a member of $\mathscr D_{5}$, then $n_{1}=1$ and $n_{h}=1$.

		\vspace{0.2cm}
		
		\noindent\textbf{Characterization II. Subdigraphs $D_{i}$ and $\widetilde{D_{i}}$ for $i=1,h$}
		
		\vspace{0.2cm}
		
		Let $\widetilde{D_{i}}$ denote $D\langle V(D_{i})\cup \{v\}\rangle$, where $i\in\{1,h\}$. We now present the structures of $D_{i}$ and $\widetilde{D_{i}}$ for $i=1,h$ as follows:
		
		\textbf{(1).} If $n_{i}=1$ for $i\in\{1,h\}$, then $v\rightarrow x$ and $y\rightarrow v$, where $V(D_{1})=\{x\}$ and $V(D_{h})=\{y\}$.

		\textbf{(2).} If $n_{i}\geq 4$ for $i\in\{1,h\}$, then $D_{i}\in \Phi^{3}_{n_{i}}(1,1)$ and $\widetilde{D_{i}}\in \Phi^{3}_{n_{i}+1}(1,1)$. This means that $|A(D_{i})|=\binom{n_{i}-1}{2}+1$ and $|A(\widetilde{D_{i}})|=\binom{n_{i}}{2}+1$, respectively. Furthermore, we have $d^{+}_{\widetilde{D_{1}}}(v)=1$, $d^{-}_{\widetilde{D_{1}}}(v)=n_{1}-2$, $d^{+}_{\widetilde{D_{h}}}(v)=n_{h}-2$ and $d^{-}_{\widetilde{D_{h}}}(v)=1$.

		\vspace{0.2cm}
		
		\noindent\textbf{Characterization III. Neighbors of $v$ in $D_{i}$ for $1\leq i\leq h-1$}

		\vspace{0.2cm}
		
		Let $U$ be the union of vertex set of $D_{i}$ for $2\leq i\leq h-1$, i.e., $U=V(D_{2})\cup\ldots\cup V(D_{h-1})$. By $\widetilde{D}\langle U\rangle$ we mean the subdigraph of $D$ induced by $U\cup \{v\}$, namely, $\widetilde{D}\langle U\rangle=D\langle U\cup \{v\}\rangle$. In addition, we denote $X=N^{+}_{\widetilde{D}\langle U\rangle}(v)$ and $Y=N^{-}_{\widetilde{D}\langle U\rangle}(v)$.
		
		\textbf{(1).} Every vertex $r\in Y$ lies in front of every vertex $s\in X$ in this sequence of vertices.
		
		\textbf{(2).} The arc set $A(V(D_{i}),V(D_{j}))$ consists of all arcs from $D_{i}$ to $D_{j}$ for $1\leq i<j\leq h$, except the following three types of arcs:

		\textbf{(2.1).} The arc $(x,y)$, where $x\in N^{+}_{\widetilde{D_{1}}}(v)$ and $y\in N^{-}_{\widetilde{D_{h}}}(v)$;
		
		\textbf{(2.2).} All arcs $(x,r)$, where $x\in N^{+}_{\widetilde{D_{1}}}(v)$ and $r\in Y$;

		\textbf{(2.3).} All arcs $(s,y)$, where $s\in X$ and $y\in N^{-}_{\widetilde{D_{h}}}(v)$.
		
		\vspace{0.1cm}
		We end in this subsection by establishing a proposition with respect to $\Phi_{n}^{3}(1,1)$ as follows:

		\begin{proposition}[\textnormal{\cite{24}}]\label{two}
			Let $D$ be a digraph of $\Phi_{n}^{3}(1,1)$ for $n\geq5$. Then $D$ contains at least two vertices of out-degree one.
		\end{proposition}
		
		Concretely, we have the following:
		
		\textbf{(1).} If $D_{h}$ has order at least $4$, i.e., $n_{h}\geq4$, then such two vertices of out-degree one belong to $V(D_{h})$ that differ from $y$, where $N^{-}_{\widetilde{D_{h}}}(v)=\{y\}$.

		\textbf{(2).} If $D_{h}$ has only one vertex $y$, then either $v$ and $y$ have out-degree one or $p$ and $y$ have out-degree one, where $V(D_{h-1})=\{p\}$. Indeed, if $p\rightarrow v$, then $v$ and $y$ are two vertices of out-degree one; otherwise, $p$ and $y$ are two vertices of out-degree one.

		\subsection{Characterization of $\Phi_{n}^{3}(2,1)$ for $n=7,8$}
		
		\noindent

		This subsection mainly elaborates the specific structure of $\Phi_{n}^{3}(2,1)$ for $n=7,8$ that will be used in a key step in the proof of Theorem~\ref{10}.

		\vspace{0.2cm}
		
		\noindent\textbf{I. Structure of $\Phi_{7}^{3}(2,1)$}
		
		\vspace{0.2cm}

		It is not difficult to deduce that $D_{n}^{3}(2,1)=\O$ for $n\leq 6$ by applying the following lemma:

		\begin{lemma}[\textnormal{\cite{4}}]\label{n/2}
			Every digraph on $n$ vertices of minimum out-degree $2$ has a cycle of length at most $\lceil n/2\rceil$.
		\end{lemma}

		Let $n\geq 2$ denote an integer and $S$ denote a subset of $\{1, 2,\ldots,n-1\}$. The \emph{circulant digraph} $C_{n}(S)$ is defined with the vertex set and arc set as follows (see, e.g., \cite{1}, Chapter $2$, p.80 for details):
		
		\textbf{(1). Vertex set:} $V(C_{n}(S))=\{1, 2,\ldots,n\}$;
		
		\textbf{(2). Arc set:} $A(C_{n}(S))=\{(i,i+j\pmod n):1\leq i\leq n,\;j\in S\}$.

		\begin{lemma}[\textnormal{\cite{24}}]\label{n=7}
			Let $D$ be a digraph of $D_{7}^{3}(2,1)$. Then $D$ is isomorphic to the circulant digraph $C_{7}(1,2)$.
		\end{lemma}

		Lemma~\ref{n=7} tells us that up to isomorphism $D_{7}^{3}(2,1)$ has only one member $C_{7}(1,2)$ and thus $\Phi_{7}^{3}(2,1)=D_{7}^{3}(2,1)=\{C_7(1,2)\}$. It is clear that the unique digraph $D$ of $\Phi_{7}^{3}(2,1)$ is $2$-regular. Namely, every vertex of $D$ has out- and in-degree two. Therefore, such $D$ contains $14$ arcs.

		\vspace{0.2cm}
		
		\noindent\textbf{II. Structure of $\Phi_{8}^{3}(2,1)$}
		
		\vspace{0.2cm}

		Let $F_{8}$ be a digraph with vertex set $V(F_{8})=\{v_{0},v_{1},\ldots,v_{7}\}$ and arc set $$A(F_{8})=\cup_{i=0}^3\{(v_{2i},v_{2i+1}),(v_{2i},v_{2i+2}),(v_{2i},v_{2i+3}),(v_{2i+1},v_{2i+2}),(v_{2i+1},v_{2i+3})\},$$
		where the subscripts are considered modulo $8$, as shown in Fig. 2. One easily checks that $F_{8}$ is a \emph{$C_{\leq 3}$-free} strong digraph on $8$ vertices with $20$ arcs. Moreover, for each $v_{i}$, if $i$ is even, then $d^{+}_{F_{8}}(v_{i})=3$ and $d^{-}_{F_{8}}(v_{i})=2$; if $i$ is odd, then $d^{+}_{F_{8}}(v_{i})=2$ and $d^{-}_{F_{8}}(v_{i})=3$.

		\begin{figure}[!hbpt]
			\begin{center}
				\includegraphics[scale=0.3]{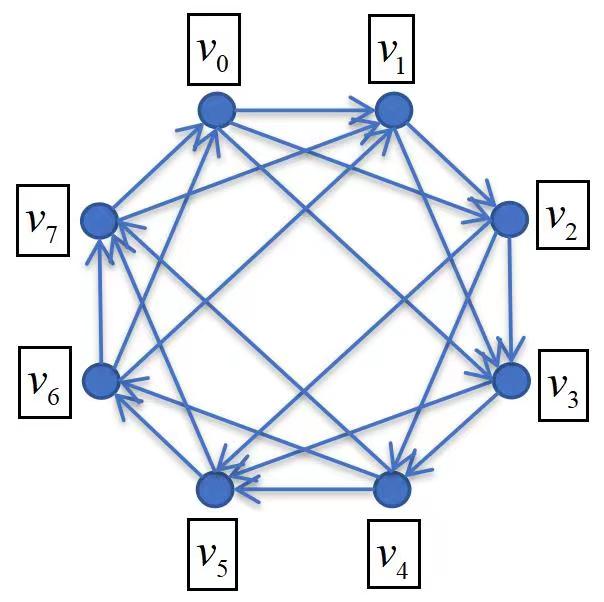}
				\caption{\small
					$F_{8}$
				}
				\label{fig:framework_of_GP}
			\end{center}
		\end{figure}
		
		Let $D$ be a digraph of $\Phi_{8}^{3}(2,1)$. In what follows we present the lemma as below, which states that, up to isomorphism, such $D$ is unique.
		
		\begin{lemma}\label{n=8}
			Let $D$ be a digraph of $\Phi_{8}^{3}(2,1)$. Then $D$ is isomorphic to $F_{8}$.
		\end{lemma}
		
		\noindent\textbf{Proof.} Since $D\in \Phi_{8}^{3}(2,1)$, 
		we have $|A(D)|=\varphi_8^3(2,1)={8-1\choose 2}-1=20$ by Theorem~\ref{upper-lower bounds}. As $\sum\limits_{v\in V(D)}{d_{D}(v)}=2|A(D)|=40$, we have $\delta(D)\leq 5$. The proof will be divided into two parts by distinguishing $\delta(D)=5$ and $\delta(D)\leq 4$.

		\vspace{0.2cm}
		\noindent\textbf{Part I. The minimum degree of $D$ is $5$, i.e., $\delta(D)=5$.}
		\vspace{0.2cm}

		It is straightforward to see that each vertex of $D$ has degree $5$. Recall that an $(\alpha,\beta)$-vertex is a vertex of out-degree $\alpha$ and in-degree $\beta$, where both $\alpha$ and $\beta$ are integers. We can derive that every vertex of $D$ is among one of $(2,3)$-vertex, $(3,2)$-vertex and $(4,1)$-vertex. Due to $\sum\limits_{v\in V(D)}{d_{D}^{+}(v)}=\sum\limits_{v\in V(D)}{d_{D}^{-}(v)}$, there must exist a $(2,3)$-vertex. 
		
		\begin{claim}
			The two out-neighbors of  any $(2,3)$-vertex are adjacent.
		\end{claim}
		
		Let $z$ be a $(2,3)$-vertex such that $z\rightarrow z_{i}$ for $i\in\{1,2\}$ and $y_{j}\rightarrow z$ for $j\in\{1,2,3\}$. In addition, we denote $V(D)\backslash (N_{D}(z)\cup \{z\})=\{u,w\}$. Suppose that $z_{1}$ and $z_{2}$ are nonadjacent. One easily checks that $z_{i}\rightarrow u$ and $z_{i}\rightarrow w$ as $d^{+}_{D}(z_{i})\geq 2$ for $i=1,2$. Hence, both $z_{1}$ and $z_{2}$ are $(2,3)$-vertices. We may assume that $y_{2}\rightarrow z_{1}$ and $y_{3}\rightarrow z_{1}$. Since $D$ is \emph{$C_{\leq 3}$-free} and $d^{+}_{D}(u)\geq2$, then $u\rightarrow y_{1}$ and $u\rightarrow w$. However, now $N^{+}_{D}(w)\subseteq \{y_{1}\}$ holds, which contradicts that $d^{+}_{D}(w)\geq2$. Thus, we obtain that the out-neighbors of every $(2,3)$-vertex are adjacent.
		
		Let $V(D)=\{v_{1},v_{2},\ldots,v_{8}\}$. We may assume without loss of generality that $v_{1}$ is a $(2,3)$-vertex such that $v_{1}\rightarrow v_{i}$ for $i=2,3$ and $v_{j}\rightarrow v_{1}$ for $j\in \{6,7,8\}$. Moreover, we may suppose that $v_{2}\rightarrow v_{3}$. As $d^{+}_{D}(v_{3})\geq 2$ and $D$ is \emph{$C_{\leq 3}$-free}, then $N^{+}_{D}(v_{3})=\{v_{4},v_{5}\}$. It yields that $v_{3}$ is a $(2,3)$-vertex and without loss of generality, we assume $v_{8} \rightarrow v_{3}$ and $v_{4}\rightarrow v_{5}$. Again, due to $d^{+}_{D}(v_{5})\geq 2$, we have $v_{5}\rightarrow v_{6}$, $v_{5}\rightarrow v_{7}$ and $v_{5}$ is also a $(2,3)$-vertex. Similarly, we suppose that $v_{6}\rightarrow v_{7}$, implying that $v_{7}$ is a $(2,3)$-vertex. Notice once again that $D$ is $\emph{$C_{\leq 3}$-free}$, we thus have $v_{7}\nrightarrow v_{i}$, where $i\in\{3,4,5,6\}$. This follows that either $v_{7}\rightarrow v_{2}$ or $v_{7}\rightarrow v_{8}$.
		
		If $v_{7}\rightarrow v_{2}$, then $v_{2}\nrightarrow v_{5}$, otherwise $v_{2},v_{5},v_{7}$ will form a cycle of length $3$, contradicting that $D$ is $\emph{$C_{\leq 3}$-free}$. This yields that $v_{2}\rightarrow v_{4}$ as $d^{+}_{D}(v_{2})\geq 2$. One easily sees that $v_{i}\nrightarrow v_{5}$ for $i\in\{1,2,6,7\}$. It is not difficult to check that $v_{i}\rightarrow v_{5}$, where $i\in\{3,4,8\}$. However, we now deduce $v_{i}\nrightarrow v_{8}$ for $1\leq i\leq 7$, which implies $d^{-}_{D}(v_{8})=0$, contradicting the fact $d^{-}_{D}(v_{8})\geq 1$. Thus, $v_{7}\nrightarrow v_{2}$ and $v_{7}\rightarrow v_{8}$ hold. As $v_{5}\rightarrow v_{7}$ and $D$ is $\emph{$C_{\leq 3}$-free}$, it is easily seen that $v_{8}\nrightarrow v_{5}$. It follows that $v_{2}\rightarrow v_{5}$ since $v_{5}$ is a $(2,3)$-vertex. Note that $v_{2}\nrightarrow v_{i}$ for $i\in\{1,6,7,8\}$. Then $d^{+}_{D}(v_{2})\leq 3$. Hence, we deduce that $d^{-}_{D}(v_{2})\geq 2$ as $d_{D}(v_{2})=5$. Moreover, $v_{i}\nrightarrow v_{2}$ holds since $D$ is $\emph{$C_{\leq 3}$-free}$, where $i\in\{3,4,5,6,7\}$. Thereby, we have $v_{8}\rightarrow v_{2}$, yielding that $v_{2}\rightarrow v_{4}$. This means that $v_{2}$ is a $(3,2)$-vertex. Analogously, $v_{4}$, $v_{6}$ and $v_{8}$ are all $(3,2)$-vertices. Consequently, we obtain that $v_{4}\rightarrow v_{6}$, $v_{4}\rightarrow v_{7}$ and $v_{6}\rightarrow v_{8}$.
		
		It is not hard to derive that such $D$ is isomorphic to $F_{8}$, as desired.

		\vspace{0.2cm}
		
		\noindent\textbf{Part II. The minimum degree of $D$ is at most $4$, i.e., $\delta(D)\leq 4$.}
		
		\vspace{0.2cm}

		Let $v$ be a vertex of minimum degree. Hence, $d_{D}(v)=\delta(D)\leq 4$. Here to proceed with this part, we consider the following two cases:
		
		\vspace{0.1cm}
		\noindent\textbf{Case 1. $D-v$ is strong, i.e., it has only one component}
		\vspace{0.1cm}
		
		Obviously, $D-v\in D_{7}^{3}(1,1)$. By Corollary~\ref{n=3},  $|A(D-v)|\le \varphi_7^3(1,1)=16$. On the other hand, since $|A(D)|=20$ and $d_{D}(v)\leq 4$, we have  $|A(D-v)|=|A(D)|-d_{D}(v)\geq 16$. It follows that $|A(D-v)|=16$ and $d_{D}(v)=4$. This means $D-v\in \Phi_{7}^{3}(1,1)$. For convenience, let $G$ stand for $D-v$. By Theorem~\ref{5}, $G\in\mathscr D_{i}$ for some $i=1,2,3,4,5$ on $7$ vertices. Let $v^{*}$ be a vertex of $G$ with minimum degree. Then $d_{G}(v^{*})\leq \lfloor\frac{2|A(G)|}7\rfloor=4=7-3$ and $G-v^{*}$ is not strong by the structure of $\Phi_{7}^{3}(1,1)$. As any digraph of $\mathscr D_{1}\cup \mathscr D_{2}$ has order greater than $7$, we obtain that $G\in \mathscr D_{3}\cup \mathscr D_{4}\cup\mathscr D_{5}$.
		
		\vspace{0.1cm}
		\textbf{Subcase 1.1. $G$ is a member of $\mathscr D_{3}$}

		By the structure of $\Phi_{7}^{3}(1,1)$,  $G-v^{*}$ contains at least three components. As $G\in \mathscr D_{3}$, by \textbf{Characterization I}, it is not hard to check that $G-v^{*}$ has three components, i.e., $D_{1},D_{2},D_{3}$, such that $n_{1}=4$ and $n_{2}=n_{3}=1$. Combined with \textbf{Characterization II}, it is clear that $y\rightarrow v^{*}$ and $p\rightarrow y$, where $V(D_{2})=\{p\}$ and $V(D_{3})=\{y\}$. Since $d_{G}(v^{*})\leq 4$, we have $p$ and $v^{*}$ are nonadjacent. Thereby, we have $d^{+}_{G}(p)=1$ and $d^{+}_{G}(y)=1$. It follows that $p\rightarrow v$ and $y\rightarrow v$ since $d^{+}_{D}(p)\geq2$ and $d^{+}_{D}(y)\geq2$. By \textbf{Characterization III}, every vertex of $D_{1}$ dominates $p$. Hence, $N^{+}_{D}(v)\subseteq \{v^{*}\}$ since $D$ is $\emph{$C_{\leq 3}$-free}$. This contradicts the fact that $d^{+}_{D}(v)\geq2$.

		\vspace{0.1cm}
		\textbf{Subcase 1.2. $G$ is a member of $\mathscr D_{4}$}

		Similarly, as $G\in \mathscr D_{4}$, it is easily seen that $G-v^{*}$ has three components, i.e., $D_{1},D_{2},D_{3}$, such that $n_{1}=n_{2}=1$ and $n_{3}=4$. This implies that $D_{3}$ is a cycle of length $4$ since $D_{3}$ is a $\emph{$C_{\leq 3}$-free}$ strong subdigraph. Suppose that $V(D_{3})=\{x_{1},x_{2},x_{3},x_{4}\}$ such that $x_{i}\rightarrow x_{i+1}$ for $1\leq i\leq3$ and $x_{4}\rightarrow x_{1}$. By \textbf{Characterization II}, we may assume $v^{*}\rightarrow x_{1}$, $v^{*}\rightarrow x_{2}$ and $x_{4}\rightarrow v^{*}$. It is easily seen that $x_{i}\rightarrow v$ as $d^{+}_{D}(x_{i})\geq2$ for $i\in\{1,2,3\}$. Using \textbf{Characterization III}, we have every vertex of $D_{1}$\;(resp., $D_{2}$) dominating $x_{2}$. One can deduce that $v\nrightarrow v^{*}$, $v\nrightarrow x_{4}$ and $v$ do not dominate any vertex of $D_{1}$\;(resp., $D_{2}$). It follows that $N^{+}_{D}(v)=\O$, contradicting the fact that $D$ is strong.
		
		\vspace{0.1cm}
		\textbf{Subcase 1.3. $G$ is a member of $\mathscr D_{5}$}

		Since $G\in \mathscr D_{5}$, then $h=6$ and $n_{i}=1$ for $i\in \{1,2,\ldots,6\}$. Clearly, $v^{*}\rightarrow x$ and $y\rightarrow v^{*}$, where $V(D_{1})=\{x\}$ and $V(D_{6})=\{y\}$. Additionally, as $G$ is a $\emph{$C_{\leq 3}$-free}$ strong subdigraph of $D$, one can deduce $v^{*}\nrightarrow p$, where $V(D_{5})=\{p\}$. \textbf{Characterization III} gives us that $p\rightarrow y$. Moreover, $y\rightarrow v$ holds since $d^{+}_{D}(y)\geq2$. If $p\nrightarrow v^{*}$, then $p\rightarrow v$. Applying \textbf{Characterization III}, it is easily seen that any vertex of $D_{i}$ dominates $p$ for $1\leq i\leq 4$. Since $D$ is $\emph{$C_{\leq 3}$-free}$, we obtain $N^{+}_{D}(v)\subseteq \{v^{*}\}$, a contradiction. If $p\rightarrow v^{*}$, then $d^{+}_{G}(v^{*})=1$, which yields that every vertex of $D_{i}$ dominates $y$ for $i\in\{2,3,4,5\}$. As $d^{+}_{D}(v^{*})\geq 2$, we obtain $v^{*}\rightarrow v$. One easily checks that $N^{+}_{D}(v)\subseteq \{x\}$ since $D$ is $\emph{$C_{\leq 3}$-free}$, which contradicts the assumption that $d^{+}_{D}(v)\geq 2$.

		\vspace{0.1cm}
		\noindent\textbf{Case 2. $D-v$ has at least two components}
		\vspace{0.1cm}
		
		Let $D_{1},D_{2},\ldots,D_{h}$ be components of $D-v$ such that all arcs between them are from $D_{i}$ to $D_{j}$, where $h\geq2$ and $1\leq i<j \leq h$. Note that $n_{h}\geq4$ since $\delta^{+}(D)\geq2$. As $D-v$ has order 7, it is sufficient to consider the following three subcases:
		
		\vspace{0.1cm}
		\textbf{Subcase 2.1. $n_{h}$=4}

		It is obvious that $D_{h}$ is a cycle of length 4 and thus every vertex of $D_{h}$ dominates $v$, namely, $d^{-}_{D}(v)\geq4$. Moreover, as $d^{+}_{D}(v)\geq 2$, we have $d_{D}(v)\geq 6$. This contradicts the fact that $d_{D}(v)\leq 4$.

		\vspace{0.1cm}
		
		\textbf{Subcase 2.2. $n_{h}$=5}

		As $D_{h}$ is a $\emph{$C_{\leq 3}$-free}$ strong subdigraph of $D$, by Corollary~\ref{n=3}, we obtain that $|A(D_{h})|\leq 7$. There are at least three vertices of $D_{h}$ of out-degree one in $D_{h}$, implying that $d^{-}_{D}(v)\geq 3$. Together with $d^{+}_{D}(v)\geq 2$, we thus have $d_{D}(v)\geq5$, which also contradicts the assumption that $d_{D}(v)\leq 4$.

		\vspace{0.1cm}
		
		\textbf{Subcase 2.3. $n_{h}$=6}

		Clearly, $D-v$ has exactly two components $D_{1}$ and $D_{2}$. Analogously, we have $|A(D_{2})|\leq 11$ by applying Corollary~\ref{n=3}. In other words, $\gamma(D_{2})\geq4$. Notice once again that $v\rightarrow x$, where $V(D_{1})=\{x\}$, and $N^{-}_{D}(v)\cap V(D_{2})\neq\O$ since $D$ is strong. Next, we will illustrate that $d_{D}(v)=4$. Suppose not, then $d_{D}(v)\leq3$, which follows that $\gamma(v,D_{2})\geq4$. In addition, since $D$ is $\emph{$C_{\leq 3}$-free}$, then $\gamma(x,D_{2})\geq1$ holds. That is, $|A(D)|\leq \binom{8}{2}-4-4-1=19$, contradicting  our assumption that $|A(D)|=20$. Thereby, we obtain that $d_{D}(v)=4$. If $d^{-}_{D\langle V(D_{2})\cup\{v\}\rangle}(v)=1$, then $d^{+}_{D\langle V(D_{2})\cup\{v\}\rangle}(v)=2$. This means that $D\langle V(D_{2})\cup\{v\}\rangle\in D_{7}^{3}(2,1)$, which contradicts the fact that, up to isomorphic, $D_{7}^{3}(2,1)$ only contains the circulant digraph $C_{7}(1,2)$ by Lemma~\ref{n=7}. Hence, we may assume that $d^{-}_{D\langle V(D_{2})\cup\{v\}\rangle}(v)\geq2$. This implies that there exist at least two vertices, denoted by $u$ and $w$, of $D_{2}$ such that $u\rightarrow v$ and $w\rightarrow v$. Since $D$ is $\emph{$C_{\leq 3}$-free}$, then $x\nrightarrow u,w$. Combined with $d_{D}(v)=4$, we obtain $|A(D)|\leq \binom{8}{2}-4-3-2=19$, contradicting the fact that $|A(D)|=20$.
		
		This completes the proof of the lemma.               \hfill $\blacksquare$

		\subsection{Several other useful lemmas}

		\noindent
		
		In this subsection, we give some helpful auxiliary lemmas that will be used in the proof of Theorem~\ref{10}.

		
		\begin{lemma}\label{minimum degree}
			Let $D$ be a digraph of $\Phi_{n}^{3}(2,1)$ for $n\geq7$. Then $\delta(D)\leq n-3$.
		\end{lemma}
		\begin{proof}
			Suppose to the contrary that $\delta(D)\ge n-2$. Then $|A(D)|\ge \frac{n(n-2)}2>{n-1\choose 2}+1$, a contradiction to Theorem~\ref{upper-lower bounds}.	
			
		\end{proof}

		In addition, we use the following lemma as the inductive basis of Theorem~\ref{10}.

		\begin{lemma}[\textnormal{\cite{24}}]\label{9}
			$\varphi_{9}^{3}(2,1)=26=\binom{9-1}{2}-2.$
		\end{lemma}

		Finally, we present a result of Thomassen \cite{25}.
		
		\begin{lemma}[\textnormal{\cite{25}}]\label{one vertex}
			Let $D$ be a strong digraph of minimum out-degree at least two. Then $D$ contains a vertex $v$ satisfying that $D-v$ is strong.
		\end{lemma}

		\section{Proof of Theorem~\ref{10}}
		
		\noindent
		
		Let $D$ be a digraph of $\Phi_{\kappa}^{3}(2,1)$. The proof will be by induction on $\kappa$. Applying Lemma~\ref{9}, we have $|A(D)|=\varphi_{\kappa}^{3}(2,1)=\binom{\kappa-1}{2}-2$ for $\kappa=9$.
		
		Assume that $\varphi_{\kappa}^{3}(2,1)=\binom{\kappa-1}{2}-2$ for all $9\leq \kappa<n$ and we consider $\kappa=n\geq10$. It suffices to prove $|A(D)|\leq\binom{n-1}{2}-2$. Suppose to the contrary that $|A(D)|>\binom{n-1}{2}-2$. By Theorem~\ref{upper-lower bounds}, we thus obtain $|A(D)|=\binom{n-1}{2}-1$. Recall that $\gamma(D)$ is the number of nonadjacent unordered pairs in $D$. Hence, $\gamma(D)=n$ holds. Lemma~\ref{minimum degree} tells us $\delta(D)\leq n-3$.
		
		In the rest of the proof, we would like to break the proof up into two parts, depending on whether the minimum degree of $D$ is equal to $n-3$.
		
		\subsection{The minimum degree of $D$ is exactly $n-3$, i.e., $\delta(D)=n-3$}

		\noindent
		
		\noindent\textbf{Proof.} Observe first that $\sum\limits_{v\in V(D)}{d_{D}(v)}=2|A(D)|$. Since $|A(D)|=\binom{n-1}{2}-1$ and $\delta(D)=n-3$, we have $n(n-3)\leq 2|A(D)|= n^{2}-3n$, which follows that every vertex has degree $n-3$. As $D\in\Phi_{n}^{3}(2,1)$, we have $\delta^+(D)\ge 2$ and $\delta^-(D)\ge 1$. Recall that an $(\alpha,\beta)$-vertex is a vertex of out-degree $\alpha$ and in-degree $\beta$, where both $\alpha$ and $\beta$ are integers. One easily sees that every vertex of $D$ is a $(t,n-t-3)$-vertex for $2\leq t\leq n-4$. Next, we claim that $D$ contains at least one $(2,n-5)$-vertex.

		\begin{claim}
			There is a $(2,n-5)$-vertex in $D$.
		\end{claim}
		\begin{proof}
			Suppose not, then $\delta^{+}(D)\geq3$. By Lemma~\ref{one vertex}, there exists a vertex $u$ such that $D-u$ is strong. One easily sees that $D-u\in D^{3}_{n-1}(2,1)$ and $n-1\geq 9$. By the induction hypothesis,  $|A(D-u)|\leq \binom{n-2}{2}-2$. Due to $d_{D}(u)=n-3$, we have $|A(D)|=|A(D-u)|+d_{D}(u)\leq \binom{n-1}{2}-3$, contradicting our assumption that $|A(D)|=\binom{n-1}{2}-1$. This proves the claim.
		\end{proof}

		We further show that each $(2,n-5)$-vertex satisfies that its out-neighbors are adjacent.
		\begin{claim}\label{adjacent}
			The out-neighbors of every $(2,n-5)$-vertex are adjacent.
		\end{claim}
		\begin{proof}
			Let $z$ be an arbitrary $(2,n-5)$-vertex. Denote by $N^{+}_{D}(z)=\{z_{1},z_{2}\}$ and $N^{-}_{D}(z)=\{y_{1},y_{2},\ldots,y_{n-5}\}$. Additionally, we denote $V(D)\backslash (N_{D}(z)\cup \{z\})=\{x_{1},x_{2}\}$. Suppose to the contrary that $z_{1}$ and $z_{2}$ are not adjacent. Since $D$ is a $\emph{$C_{\leq 3}$-free}$ digraph with $\delta^{+}(D)\geq 2$, $z_{i}\rightarrow x_{j}$ for $i,j\in\{1,2\}$. It follows that both $z_{1}$ and $z_{2}$ are $(2,n-5)$-vertices. Without loss of generality, we may assume that $y_{i}\rightarrow z_{1}$, where $i\in\{2,3,\ldots,n-5\}$. As $d^{+}_{D}(x_{1})\geq2$ and $D$ is $\emph{$C_{\leq 3}$-free}$, we thus obtain $x_{1}\rightarrow x_{2}$ and $x_{1}\rightarrow y_{1}$. This  yields  $N^{+}_{D}(x_{2})\subseteq\{y_{1}\}$, which contradicts the fact that $d^{+}_{D}(x_{2})\geq 2$. As a consequence, the claim follows.
		\end{proof}

		Label the vertices of $D$ by $v_{1},v_{2},\ldots,v_{n}$. We may assume without loss of generality that $v_{1}$ is a $(2,n-5)$-vertex satisfying that $v_{1}\rightarrow v_{i}$ for $i=2,3$ and $v_{j}\rightarrow v_{1}$ for $j\in\{6,7,\ldots,n\}$. By the previous claim, we may suppose $v_{2}\rightarrow v_{3}$. One easily sees that $v_{3}\rightarrow v_{i}$ for $i=4,5$ since $D$ is $\emph{$C_{\leq 3}$-free}$ and $d^{+}_{D}(v_{3})\geq 2$. This implies that $v_{3}$ is a $(2,n-5)$-vertex and we may assume that $v_{i}\rightarrow v_{3}$ for $8\leq i\leq n$. By Claim~\ref{adjacent}, we have $v_{4}$ and $v_{5}$ are adjacent and without loss of generality $v_{4}\rightarrow v_{5}$ holds. As $D$ is $\emph{$C_{\leq 3}$-free}$, we obtain $v_{5}\nrightarrow v_{i}$ for $i\in\{1,2,3,4\}$, and $v_{5}\nrightarrow v_{j}$ for  $j\in\{8,9,\ldots,n\}$. It is clear that $N^{+}_{D}(v_{5})=\{v_{6},v_{7}\}$ since $d^{+}_{D}(v_{5})\geq 2$, yielding that $v_{5}$ is a $(2,n-5)$-vertex. The same reason shows that $v_{6}$, and $v_{7}$ are adjacent and we may assume $v_{6}\rightarrow v_{7}$. We now deal with the connection of $v_{2}$ and $v_{5}$ by distinguishing the following two cases:

		\vspace{0.1cm}
		
		\noindent\textbf{Case 1. $v_{2}\rightarrow v_{5}$}
		\vspace{0.1cm}
		
		Since $v_{5}$ is a $(2,n-5)$-vertex, it is not difficult to see that $v_{5}$ is dominated by $n-8$ vertices among $v_{8},v_{9},\ldots,v_{n}$. Suppose without loss of generality that $v_{i}\rightarrow v_{5}$ for $9\leq i\leq n$. One easily checks that $v_{7}\nrightarrow v_{i}$ and $v_{7}\nrightarrow v_{j}$ since $D$ is $\emph{$C_{\leq 3}$-free}$, where $i\in\{2,3,4\}$ and $j\in\{9,10,\ldots,n\}$. It follows that $v_{7}\rightarrow v_{8}$ because $d_D^+(v_7)\ge 2$. In addition, $v_{6}$ and $v_{2}$\;(resp., $v_{3}$) are nonadjacent as $D$ is $\emph{$C_{\leq 3}$-free}$. Because $d_{D}(v_{6})=n-3$, $v_{6}$ is adjacent to every vertex except $v_{2}$ and $v_{3}$. As $v_{6}\rightarrow v_{7}\rightarrow v_{8}$, then $v_{6}\rightarrow v_{8}$ holds. Moreover, since $D$ is $\emph{$C_{\leq 3}$-free}$, any vertex in $\{v_{9},v_{10},\ldots,v_{n}\}$ dominates $v_{6}$. Let $R=\{v_{1},v_{2},\ldots,v_{8}\}$ and $S=\{v_{9},v_{10},\ldots,v_{n}\}$. Since $D$ is $\emph{$C_{\leq 3}$-free}$, we have  $A(R,S)=\O$, contradicting the fact that $D$ is a strong digraph.

		\vspace{0.1cm}
		
		\noindent\textbf{Case 2. $v_{2}\nrightarrow v_{5}$}
		\vspace{0.1cm}

		As $v_{5}$ is a $(2,n-5)$-vertex, then $v_{i}\rightarrow v_{5}$ for $8\leq i\leq n$. Let $R^{*}=\{v_{1},v_{2},\ldots,v_{7}\}$ and $S^{*}=\{v_{8},v_{9},\ldots,v_{n}\}$. Analogously, since $D$ is $\emph{$C_{\leq 3}$-free}$, we have $A(R^{*},S^{*})=\O$. Whereas this contradicts our assumption that $D$ is strong.  \hfill $\blacksquare$

		\subsection{The minimum degree of $D$ is at most $n-4$, i.e., $\delta(D)\leq n-4$.}
		
		\noindent
		
		\noindent\textbf{Proof.} Let $w$ be a vertex of $D$ with minimum degree. Then $d_{D}(w)=\delta(D)\leq n-4$. In the remaining proof in this subsection, we address  $D-w$ by distinguishing two parts as follows, namely, $D-w$ is strong and $D-w$ is not strong:
		
		\vspace{0.2cm}
		
		\noindent\textbf{Part I. $D-w$ is strong}
		
		\vspace{0.2cm}
		
		Notice first that $D-w\in D_{n-1}^{3}(1,1)$. Hence, $|A(D-w)|\leq\varphi_{n-1}^3(1,1)=\binom{n-2}{2}+1$. Together with $d_{D}(w)\leq n-4$, we have $|A(D)|=|A(D-w)|+d_{D}(w)\leq \binom{n-1}{2}-1$. Recall that $|A(D)|=\binom{n-1}{2}-1$, which implies that both $|A(D-w)|=\binom{n-2}{2}+1$ and $d_{D}(w)=n-4$ hold. For convenience, we denote $G=D-w$.
		One easily sees that $G\in \Phi^{3}_{n-1}(1,1)$. In the rest of proof in this part, we will use the structure of digraphs of $\Phi^{3}_{n-1}(1,1)$. By applying Theorem~\ref{5}, one can separate $G$ into five families of digraphs on $n-1$ vertices, namely, \{$\mathscr D_{1},\mathscr D_{2},\mathscr D_{3},\mathscr D_{4},\mathscr D_{5}$\}, as shown in Fig.1\;(see Subsection 2.2 for details).
		
		Let $v$ be a vertex of $G$ such that $d_{G}(v)\leq n-4$ and $G-v$ has $h$ components, denoted by $D_{1},D_{2},\ldots,D_{h}$, satisfying that there is no arc from $D_{j}$ to $D_{i}$ for $1\leq i<j\leq h$, where $h\geq2$. Denote by $|V(D_{i})|=n_{i}$ for $i\in\{1,2,\ldots,h\}$. We divide the discussion into five cases as follows:
		
		\vspace{0.1cm}
		\noindent\textbf{Case 1. $G\in \mathscr D_{1}$}
		\vspace{0.1cm}

		Note first that $G-v$ contains exactly two components $D_{1}$ and $D_{2}$ such that $n_{1}\geq4$ and $n_{2}\geq4$. By \textbf{Characterization II}, we have $d^{+}_{D\langle V(D_{2})\cup \{v\}\rangle}(v)=n_{2}-2\geq 2$ and $d^{-}_{D\langle V(D_{2})\cup \{v\}\rangle}(v)=1$. Moreover, $D\langle V(D_{i})\cup \{v\}\rangle\in \Phi^{3}_{n_{i}+1}(1,1)$ for $i=1,2$. Applying Proposition~\ref{two}, there are two vertices of out-degree ones that differ from $y$ of $D_{2}$, where $N^{-}_{D\langle V(D_{2})\cup \{v\}\rangle}(v)=\{y\}$. Hence, one easily finds a vertex $u\in N^{+}_{D\langle V(D_{2})\cup \{v\}\rangle}(v)$ such that $d^{+}_{G}(u)=1$. Since $d^{+}_{D}(u)\geq 2$, we thus obtain $u\rightarrow w$. \textbf{Characterization III} gives us that every vertex of $D_{1}$ dominates $u$. It is not difficult to check that $w$ does not dominate any vertex of $D_{1}$ and $w\nrightarrow v$ as $D$ is $\emph{$C_{\leq 3}$-free}$. It follows that $N^{+}_{D}(w)\subseteq V(D_{2})$. Denote by $D_{2}^{*}=D\langle V(D_{2})\cup\{v,w\}\rangle$. One easily sees that $D_{2}^{*}\in D^{3}_{n_{2}+2}(2,1)$. Since $d^{-}_{D_{2}^{*}}(v)=1$, by Lemma~\ref{n=7}, we obtain $n_{2}\geq6$. We only need to consider the following two subcases:
		
		\vspace{0.1cm}
		\textbf{Subcase 1.1. $n_{2}=6$}

		Combining Lemma~\ref{n=8} and $d^{-}_{D_{2}^{*}}(v)=1$, it is not difficult to check that $D_{2}^{*}\notin \Phi^{3}_{8}(2,1)$. This implies that $\gamma(D_{2}^{*})\geq 9$. Additionally, $\gamma(D\langle V(D_{1})\cup \{v\}\rangle)=n_{1}-1$ and $\gamma(D_{1},D_{2})=1$. One can deduce that $\gamma(D)\geq \gamma(D\langle V(D_{1})\cup \{v\}\rangle)+\gamma(D_{2}^{*})+\gamma(D_{1},D_{2})\geq n+1$, contradicting the fact that $\gamma(D)=n$.

		\vspace{0.1cm}
		\textbf{Subcase 1.2. $n_{2}\geq7$}

		Obviously, we have $|V(D_{2}^{*})|\geq9$. By the inductive hypothesis, $|A(D_{2}^{*})|\leq \varphi^{3}_{n_{2}+2}(2,1)=\binom{n_{2}+1}{2}-2$ holds, i.e., $\gamma(D_{2}^{*})\geq n_{2}+3$. Together with $\gamma(D\langle V(D_{1})\cup \{v\}\rangle)=n_{1}-1$ and $\gamma(D_{1},D_{2})=1$, it is easily seen that $\gamma(D)\geq \gamma(D\langle V(D_{1})\cup \{v\}\rangle)+\gamma(D_{2}^{*})+\gamma(D_{1},D_{2})\geq n+1$, also a contradiction.

		\vspace{0.1cm}
		\noindent\textbf{Case 2. $G\in \mathscr D_{2}$}
		\vspace{0.1cm}

		Indeed, $G-v$ has $h\geq3$ components $D_{1},D_{2},\ldots,D_{h}$ satisfying that $n_{i}\geq4$ for $i\in\{1,h\}$ and $n_{j}=1$ for $2\leq j\leq h-1$. Moreover, $d^{+}_{D\langle V(D_{h})\cup \{v\}\rangle}(v)=n_{h}-2\geq 2$ and $d^{-}_{D\langle V(D_{h})\cup \{v\}\rangle}(v)=1$ hold due to \textbf{Characterization II}. By  Proposition~\ref{two}, one easily checks that $D_{h}$ contains two vertices of out-degree ones that are different from $y$, where $N^{-}_{D\langle V(D_{h})\cup \{v\}\rangle}(v)=\{y\}$. Consequently, it is not difficult to find a vertex $u\in N^{+}_{D\langle V(D_{h})\cup \{v\}\rangle}(v)$ satisfying that $d^{+}_{G}(u)=1$, which follows that $u\rightarrow w$. By \textbf{Characterization III}, we obtain that each vertex of $D_{i}$ dominates $u$, where $i\in\{1,2,\ldots,h-1\}$. It yields that $w$ does not dominate any vertex of $D_{i}$ for $1\leq i\leq h-1$ and $w\nrightarrow v$ since $D$ is $\emph{$C_{\leq 3}$-free}$. This means $N^{+}_{D}(w)\subseteq V(D_{h})$. Denote by $D_{h}^{*}=D\langle V(D_{h})\cup\{v,w\}\rangle$. It is clear that $D_{h}^{*}\in D^{3}_{n_{h}+2}(2,1)$. Let $K=D\langle V(D_{2})\cup V(D_{3})\cup \ldots \cup V(D_{h-1})\rangle$. Again by using \textbf{Characterization III}, we have $\gamma(v,K)+\gamma(D_{1},K)+\gamma(K,D_{h})\geq n-n_{1}-n_{h}-2$. In addition, $\gamma(D\langle V(D_{1})\cup \{v\}\rangle)=n_{1}-1$ and $\gamma(D_{1},D_{h})=1$. Combining Lemma~\ref{n=7} and $d^{-}_{D_{h}^{*}}(v)=1$, one easily sees that $n_{h}\geq6$. Analogously, we also consider the subcases $n_{h}=6$ and $n_{h}\geq7$.
		
		\vspace{0.1cm}
		\textbf{Subcase 2.1. $n_{h}=6$}

		Since $d^{-}_{D_{h}^{*}}(v)=1$, by Lemma~\ref{n=8}, it is easily seen that $D_{h}^{*}\notin \Phi^{3}_{8}(2,1)$. That is, $\gamma(D_{h}^{*})\geq 9$ holds. Thus, we have
		\begin{align*}
			\gamma(D)
			&\geq\gamma(D\langle V(D_{1})\cup \{v\}\rangle)+\gamma(D_{h}^{*})+\gamma(v,K)+\gamma(D_{1},K)+\gamma(K,D_{h})+\gamma(D_{1},D_{h})\\
			&\geq n_{1}-1+9+n-n_{1}-6-2+1\\
			&=n+1,
		\end{align*}
		a contradiction to the fact that $\gamma(D)=n$.

		\vspace{0.1cm}
		\textbf{Subcase 2.2. $n_{h}\geq7$}

		Note once again that $|V(D_{h}^{*})|\geq9$. By the inductive hypothesis, we have $|A(D_{h}^{*})|\leq \varphi^{3}_{n_{h}+2}(2,1)=\binom{n_{h}+1}{2}-2$, namely, $\gamma(D_{h}^{*})\geq n_{h}+3$. Similarly, the following holds:
		\begin{align*}
			\gamma(D)
			&\geq\gamma(D\langle V(D_{1})\cup \{v\}\rangle)+\gamma(D_{h}^{*})+\gamma(v,K)+\gamma(D_{1},K)+\gamma(K,D_{h})+\gamma(D_{1},D_{h})\\
			&\geq n_{1}-1+n_{h}+3+n-n_{1}-n_{h}-2+1\\
			&=n+1,
		\end{align*}
		a contradiction too.

		\vspace{0.1cm}
		\noindent\textbf{Case 3. $G\in \mathscr D_{3}$}
		\vspace{0.1cm}

		In fact, $G-v$ contains $h\geq 3$ components and $n_{i}=1$ for $2\leq i\leq h$. Notice that $y\rightarrow v$ and $p\rightarrow y$, where $V(D_{h-1})=\{p\}$ and $V(D_{h})=\{y\}$. Additionally, $v\nrightarrow p$ since $G$ is a $\emph{$C_{\leq 3}$-free}$ strong subdigraph of $D$. If $p\nrightarrow v$, then $d^{+}_{G}(p)=1$, implying that $p\rightarrow w$ as $d^{+}_{D}(p)\geq2$. On the other hand, by \textbf{Characterization III}, every vertex of $D_{i}$ dominates $p$, where $i\in\{1,2,\ldots,h-2\}$. One easily sees that $w$ does not dominate any vertex of $D_{i}$ for $i\in\{1,2,\ldots,h-2\}$ since $D$ is $\emph{$C_{\leq 3}$-free}$. This follows that $N^{+}_{D}(w)\subseteq\{v\}$, contradicting the fact that $d^{+}_{D}(w)\geq2$. If $p\rightarrow v$, then $d^{+}_{G}(v)=1$ and each vertex of $D_{i}$ dominates $y$ except $x$ by \textbf{Characterization III}, where $1\leq i\leq h-1$ and $N^{+}_{D\langle V(D_{1})\cup \{v\}\rangle}(v)=\{x\}$. Hence, $v\rightarrow w$ holds as $d^{+}_{D}(v)\geq 2$. Moreover, $N^{+}_{D}(w)\subseteq\{x\}$ holds since $D$ is $\emph{$C_{\leq 3}$-free}$. However, this contradicts our assumption that $d^{+}_{D}(w)\geq2$.

		\vspace{0.1cm}
		\noindent\textbf{Case 4. $G\in \mathscr D_{4}$}
		\vspace{0.1cm}

		Clearly, $V(D_{1})=\{x\}$ and $v\rightarrow x$. Analogous to the proof of Case 2, one easily finds a vertex $u\in N^{+}_{D\langle V(D_{h})\cup \{v\}\rangle}(v)$ such that $d^{+}_{G}(u)=1$ and  $u\rightarrow w$. Moreover, each vertex of $D_{i}$ dominates such $u$, where $1\leq i\leq h-1$. It can be deduced that $N^{+}_{D}(w)\subseteq V(D_{h})$ since $D$ is $\emph{$C_{\leq 3}$-free}$. Set $D_{h}^{*}=D\langle V(D_{h})\cup\{v,w\}\rangle$. Similarly, we have $n_{h}\geq 6$ by Lemma~\ref{n=7}. Denote by $K=D\langle V(D_{2})\cup V(D_{3})\cup \ldots \cup V(D_{h-1})\rangle$. Then $\gamma(v,K)+\gamma(D_{1},K)+\gamma(K,D_{h})\geq n-n_{h}-3$ and $\gamma(x,D_{h})=1$. It is sufficient to consider the subcases $n_{h}=6$ and $n_{h}\geq7$:

		\vspace{0.1cm}
		\textbf{Subcase 4.1. $n_{h}=6$}

		Due to $d^{-}_{D_{h}^{*}}(v)=1$, by Lemma~\ref{n=8}, it is not hard to check that $D_{h}^{*}\notin \Phi^{3}_{8}(2,1)$. This implies $\gamma(D_{h}^{*})\geq 9$. As a consequence, we conclude that
		\begin{align*}
			\gamma(D)
			&\geq\gamma(D_{h}^{*})+\gamma(v,K)+\gamma(x,K)+\gamma(K,D_{h})+\gamma(x,D_{h})\\
			&\geq 9+n-9+1\\
			&=n+1.
		\end{align*}
		We contradicts the fact that $\gamma(D)=n$.

		\vspace{0.1cm}
		\textbf{Subcase 4.2. $n_{h}\geq7$}

		Observe first that $|V(D_{h}^{*})|\geq9$. By the inductive hypothesis, we can derive that $\gamma(D_{h}^{*})\geq n_{h}+3$. Therefore, the following holds:
		\begin{align*}
			\gamma(D)
			&\geq\gamma(D_{h}^{*})+\gamma(v,K)+\gamma(x,K)+\gamma(K,D_{h})+\gamma(x,D_{h})\\
			&\geq n_{h}+3+n-n_{h}-3+1\\
			&=n+1.
		\end{align*}
		This also contradicts the fact that $\gamma(D)=n$.

		\vspace{0.1cm}
		\noindent\textbf{Case 5. $G\in \mathscr D_{5}$}
		\vspace{0.1cm}

		Recall that each $D_{i}$ contains a unique vertex, namely, $n_{i}=1$ for $i\in\{1,2,\ldots,h\}$. Note once again that $p\rightarrow y\rightarrow v\rightarrow x$ by \textbf{Characterization II}, where $V(D_{1})=\{x\}$, $V(D_{h-1})=\{p\}$ and $V(D_{h})=\{y\}$. In addition, $v\nrightarrow p$ as $G$ is a $\emph{$C_{\leq 3}$-free}$ strong subdigraph of $D$. Combining $d^{+}_{G}(y)=1$ and $d^{+}_{D}(y)\geq2$, it is clear that $y\rightarrow w$. If $p\nrightarrow v$, then $p\rightarrow w$ since $d^{+}_{D}(p)\geq2$, and \textbf{Characterization III} tells us that any vertex of $D_{i}$ dominates $p$ for $i\in\{1,2,\ldots,h-2\}$. This yields that $N^{+}_{D}(w)\subseteq\{v\}$, a contradiction. If $p\rightarrow v$, then \textbf{Characterization III} shows that $d^{+}_{G}(v)=1$, implying that $v\rightarrow w$. Moreover, every vertex of $D_{i}$ dominates $y$ for $i\in\{2,\ldots,h-2\}$. Therefore, one can derive $N^{+}_{D}(w)\subseteq\{x\}$, which contradicts the fact that $d^{+}_{D}(w)\geq2$.

		\vspace{0.2cm}
		
		\noindent\textbf{Part II. $D-w$ is not strong}
		
		\vspace{0.2cm}

		Indeed, $D-w$ contains $h\geq 2$ components. These components have an acyclic ordering $D_{1},D_{2},\ldots,D_{h}$ such that all arcs between them are from $D_{i}$ to $D_{j}$ for $1\leq i<j\leq h$. Additionally, let $n_{i}$ denote the number of vertices of $D_{i}$, where $i\in\{1,2,\ldots,h\}$. It is straightforward to see that either $n_{i}=1$ or $n_{i}\geq4$ since $D$ is $\emph{$C_{\leq 3}$-free}$. Due to $\delta^{+}(D)\geq 2$, then $n_{h}\geq4$ holds.

		From now on, we shall use the following notation. Let $G_{i}=D\langle V(D_{i})\cup\{w\}\rangle$ for $i\in\{1,2,\ldots,h\}$ and $U=D\langle V(D_{2})\cup V(D_{3})\cup\ldots\cup V(D_{h-1})\rangle$. As $D$ is strong, one easily sees that $N^{+}_{D}(w)\cap V(D_{1})\neq{\O}$ and $N^{-}_{D}(w)\cap V(D_{h})\neq{\O}$. We write $K_{i}$\;(resp., $L_{i}$) for $N^{+}_{D}(w)\cap V(D_{i})$\;(resp., $N^{-}_{D}(w)\cap V(D_{i})$), and we denote $|K_{i}|=k_{i}$\;(resp., $|L_{i}|=\ell_{i}$). where $i\in\{1,2,\ldots,h\}$. Clearly, both $k_{1}$ and $\ell_{h}$ are positive integers since $D$ is strong. One easily checks that $\gamma(D_{1},D_{h})\geq k_{1}\ell_{h}$ as $D$ is $\emph{$C_{\leq 3}$-free}$. Moreover, we denote $X=N^{+}_{D}(w)\cap V(U)$, $Y=N^{-}_{D}(w)\cap V(U)$ and $Z=V(U)\backslash (X\cup Y)$. Clearly, $X,Y,Z$ are mutually disjoint and $V(U)=X\cup Y\cup Z$. Let $x,y,z$ be cardinalities of $X,Y,Z$, respectively. Thus, $x+y+z=n-n_{1}-n_{h}-1$ holds. As $D$ is $\emph{$C_{\leq 3}$-free}$, one checks that $A(X,L_{h})=\O$ and $A(K_{1},Y)=\O$. It is not difficult to see that $\gamma(w,U)=z$, $\gamma(D_{1},U)\geq k_{1}y$ and $\gamma(U,D_{h})\geq x\ell_{h}$. Therefore, one easily deduces that $\gamma(w,U)+\gamma(D_{1},U)+\gamma(U,D_{h})\geq z+k_{1}y+x\ell_{h}\geq n-n_{1}-n_{h}-1$. We define $\gamma^{*}$ to be $\gamma(w,U)+\gamma(D_{1},U)+\gamma(U,D_{h})$, and it is obvious that $\gamma^{*}\geq z+k_{1}y+x\ell_{h}\geq n-n_{1}-n_{h}-1$.

		To make the proof readable, we consider the following four cases:
		
		\vspace{0.1cm}
		
		\noindent\textbf{Case 1. Both $G_{1}$ and $G_{h}$ are strong}
		
		\vspace{0.1cm}
		
		Notice first that $G_{1}\in D_{n_{1}+1}^{3}(1,1)$. It follows that $\gamma(G_{1})\geq n_{1}-1$ by Corollary~\ref{n=3}. In addition, $K_{h}\neq\O$ as $G_{h}$ is a strong subdigraph of $D$. Here to proceed with this case, we consider the following subcases:

		\vspace{0.1cm}
		\textbf{Subcase 1.1. $k_{h}\geq2$}
		\vspace{0.1cm}
		
		Since $\delta^{+}(D)\geq2$ and there is no arc from $D_{h}$ to $D_{i}$ for $i\in\{1,2,\ldots,h-1\}$, one easily derives that $G_{h}\in D_{n_{h}+1}^{3}(2,1)$. By Lemma~\ref{n/2}, $n_{h}\geq 6$ holds since $G_{h}$ is a $\emph{$C_{\leq 3}$-free}$ strong digraph of $D$.
		
		\vspace{0.1cm}
		\textbf{(1). $n_{h}=6$}. By applying Lemma~\ref{n=7}, $G_{h}$ is isomorphic to the circulant digraph $C_{7}(1,2)$. It yields that $\ell_{h}=2$ and $\gamma(G_{h})=7$. Additionally, $\gamma(D_{1},D_{h})\geq k_{1}\ell_{h}=2k_{1}\geq2$. Combining with $\gamma^{*}\geq n-n_{1}-n_{h}-1=n-n_{1}-7$, we have:
		\begin{align*}
			\gamma(D)
			&=\gamma(G_{1})+\gamma(G_{h})+\gamma^{*}+\gamma(D_{1},D_{h})\\
			&\geq n_{1}-1+7+n-n_{1}-7+2\\
			&=n+1.
		\end{align*}
		
		This inequality contradicts our assumption that $\gamma(D)=n$.

		\vspace{0.1cm}
		\textbf{(2). $n_{h}=7$}. We now break the proof up into the following two cases, depending on whether $G_{h}\in \Phi_{8}^{3}(2,1)$ or not.
		
		\vspace{0.1cm}
		\textbf{(2.1)}. If $G_{h}\in \Phi_{8}^{3}(2,1)$, then $G_{h}$ is isomorphic to $F_{8}$, which implies that $\ell_{h}\geq 2$ and $\gamma(G_{h})=8$. Analogously, $\gamma(D_{1},D_{h})\geq2$ holds. Together with $\gamma^{*}\geq n-n_{1}-n_{h}-1=n-n_{1}-8$, the following holds:
		\begin{align*}
			\gamma(D)
			&=\gamma(G_{1})+\gamma(G_{h})+\gamma^{*}+\gamma(D_{1},D_{h})\\
			&\geq n_{1}-1+8+n-n_{1}-8+2\\
			&=n+1.
		\end{align*}
		
		This also contradicts the fact that $\gamma(D)=n$.
		
		\vspace{0.1cm}
		\textbf{(2.2)}. If $G_{h}\notin \Phi_{8}^{3}(2,1)$, then $\gamma(G_{h})\geq9$. Observe once again that $\gamma(D_{1},D_{h})\geq k_{1}\ell_{h}\geq1$ and $\gamma^{*}\geq n-n_{1}-n_{h}-1=n-n_{1}-8$. Consequently, we have
		\begin{align*}
			\gamma(D)
			&=\gamma(G_{1})+\gamma(G_{h})+\gamma^{*}+\gamma(D_{1},D_{h})\\
			&\geq n_{1}-1+9+n-n_{1}-8+1\\
			&=n+1.
		\end{align*}
		
		This also leads to a contradiction.

		\vspace{0.1cm}
		\textbf{(3). $n_{h}\geq8$}. Obviously, $9\leq n_{h}+1<n$ and $G_{h}\in D_{n_{h}+1}^{3}(2,1)$. By the induction hypothesis, we obtain $|A(G_{h})|\leq \varphi_{n_{h}+1}^{3}(2,1)=\binom{n_{h}}{2}-2$. That is, $\gamma(G_{h})\geq n_{h}+2$. Note that $\gamma(D_{1},D_{h})\geq1$ and $\gamma^{*}\geq n-n_{1}-n_{h}-1$. We thus obtain that
		\begin{align*}
			\gamma(D)
			&=\gamma(G_{1})+\gamma(G_{h})+\gamma^{*}+\gamma(D_{1},D_{h})\\
			&\geq n_{1}-1+n_{h}+2+n-n_{1}-n_{h}-1+1\\
			&=n+1.
		\end{align*}
		
		The above inequality also contradicts the assumption that $\gamma(D)=n$.

		\vspace{0.1cm}
		
		\textbf{Subcase 1.2. $k_{h}=1$ and $\ell_{h}\leq 2$}
		
		Recall that $D_{h}\in D^{3}_{n_{h}}(1,1)$ and we thus have $|A(D_{h})|\leq \varphi^{3}_{n_{h}}(1,1)=\binom{n_{h}-1}{2}+1$, yielding that $\gamma(D_{h})\geq n_{h}-2$. Next, we demonstrate that $n_{h}\geq 6$. Suppose to the contrary that $n_{h}\leq 5$, then $n_{h}$ is equal to $4$ or $5$. If $n_{h}=4$, then $D_{h}$ is a cycle of length $4$ such that every vertex of $D_{h}$ has out-degree one in $D_{h}$. This implies that each vertex of $D_{h}$ dominates $w$ since there is no arc from $D_{i}$ to $D_{h}$ for $1\leq i\leq h-1$ and $\delta^{+}(D)\geq2$. Ihis contradicts that $G_{h}$ being a strong subdigraph of $D$. If $n_{h}=5$, then we have $D_{h}\in D^{3}_{5}(1,1)$. Hence, $|A(D_{h})|\leq \varphi^{3}_{5}(1,1)=7$. It is not hard to check that $D_{h}$ contains at least three vertices of out-degree one. This yields $\ell_{h}\geq3$ since $\delta^{+}(D)\geq2$, contradicting our assumption that $\ell_{h}\leq2$. Consequently, we obtain that $n_{h}\geq 6$. Moreover, $\gamma(D_{1},D_{h})\geq \ell_{h}$ and $\gamma(w,D_{h})=n_{h}-\ell_{h}-1$. One can deduce that
		\begin{align*}
			\gamma(D)
			&=\gamma(G_{1})+\gamma(D_{h})+\gamma^{*}+\gamma(D_{1},D_{h})+\gamma(w,D_{h})\\
			&\geq n_{1}-1+n_{h}-2+n-n_{1}-n_{h}-1+\ell_{h}+n_{h}-\ell_{h}-1\\
			&=n+n_{h}-5\\
			&\geq n+1,
		\end{align*}
		contradicts the assumption that $\gamma(D)=n$.

		\vspace{0.1cm}
		
		\textbf{Subcase 1.3. $k_{h}=1$ and $\ell_{h}\geq 3$}
		
		Analogous to the previous discussion, one easily derives that $n_{h}\geq 5$. In addition, it is obvious that $\gamma(D_{1},D_{h})\geq k_{1}\ell_{h}\geq 3$. Recall that $G_{h}\in D^{3}_{n_{h}+1}(1,1)$ and $n_{h}+1\geq6$. Notice first that $G_{h}$ has only one vertex of out-degree one, by Proposition~\ref{two}, we have $G_{h}\notin \Phi^{3}_{n_{h}+1}(1,1)$. Applying Corollary~\ref{n=3}, one easily sees that $|A(G_{h})|\leq \binom{n_{h}}{2}$, i.e., $\gamma(G_{h})\geq n_{h}$. Combining $\gamma(G_{1})\geq n_{1}-1$ and $\gamma^{*}\geq n-n_{1}-n_{h}-1$, the following holds:
		\begin{align*}
			\gamma(D)
			&=\gamma(G_{1})+\gamma(G_{h})+\gamma^{*}+\gamma(D_{1},D_{h})\\
			&\geq n_{1}-1+n_{h}+n-n_{1}-n_{h}-1+3\\
			&=n+1,
		\end{align*}
		a contradiction to the fact that $\gamma(D)=n$ too.

		\vspace{0.1cm}
		
		\noindent\textbf{Case 2. $G_{1}$ is strong and $G_{h}$ is not strong}
		
		\vspace{0.1cm}

		It is straightforward that $k_{h}=0$ and $\ell_{h}\geq1$, which follows that $\gamma(w,D_{h})=n_{h}-\ell_{h}$. Since $d^{+}_{D}(w)\geq 2$, we have $|N^{+}_{D}(w)\cap (V(D_{1})\cup V(U))|\geq2$. We now consider the sum of $\gamma(D_{1},D_{h})$ and $\gamma^{*}$. If
		$k_{1}\geq 2$, then $\gamma(D_{1},D_{h})\geq 2\ell_{h}$. Hence, we obtain $\gamma(D_{1},D_{h})+\gamma^{*}\geq 2\ell_{h}+n-n_{1}-n_{h}-1$. If $k_{1}=1$, then there is a vertex $t$ of $U$ such that $w\rightarrow t$. Consequently, $\gamma(t,D_{h})\geq \ell_{h}$ since $D$ is $\emph{$C_{\leq 3}$-free}$. It is not difficult to check that $\gamma(D_{1},D_{h})+\gamma^{*}\geq \ell_{h}+n-n_{1}-n_{h}-2+\ell_{h}=2\ell_{h}+n-n_{1}-n_{h}-2$. In summary, $\gamma(D_{1},D_{h})+\gamma^{*}\geq 2\ell_{h}+n-n_{1}-n_{h}-2$ holds. Similarly, as shown above, $\gamma(G_{1})\geq n_{1}-1$ and $\gamma(D_{h})\geq n_{h}-2$ hold. Combining the above conclusions, we deduce that
		\begin{align*}
			\gamma(D)
			&=\gamma(G_{1})+\gamma(D_{h})+\gamma^{*}+\gamma(D_{1},D_{h})+\gamma(w,D_{h})\\
			&\geq n_{1}-1+n_{h}-2+2\ell_{h}+n-n_{1}-n_{h}-2+n_{h}-\ell_{h}\\
			&=n+\ell_{h}+n_{h}-5.
		\end{align*}
		
		Recall that $n_{h}\geq 4$. Next, we state that $\ell_{h}+n_{h}\geq7$. Obviously, if $n_{h}\geq 6$, then we are done. Due to $\delta^{+}(D)\geq2$, it is easily seen that $\ell_{h}=4$ if $n_{h}=4$, yielding that $\ell_{h}+n_{h}=8$. Hence, we assume $n_{h}=5$. One easily checks that $D_{h}\in D^{3}_{5}(1,1)$, implying that $|A(D_{h})|\leq 7$. There exist at least three vertices of out-degree one in $D_{h}$ and we have $\ell_{h}\geq3$. Therefore, we obtain $\ell_{h}+n_{h}\geq8$. As a consequence, we obtain that $\ell_{h}+n_{h}\geq7$. This follows $\gamma(D)\geq n+2$, a contradiction.

		\vspace{0.1cm}
		
		\noindent\textbf{Case 3. $G_{1}$ is not strong and $G_{h}$ is strong}
		
		\vspace{0.1cm}

		Recall that $D_{1}$ contains either exactly one vertex or at least $4$ vertices since $D$ is a $\emph{$C_{\leq 3}$-free}$ strong digraph. Here to proceed with this case, we present the proof by distinguishing the following two subcases, namely, $n_{1}=1$ and $n_{1}=4$:

		\vspace{0.1cm}
		
		\textbf{Subcase 3.1. $D_{1}$ has a unique vertex, namely, $n_{1}=1$}

		Clearly, we have $k_{1}=1$. Note that $K_{h}\neq\O$ since $G_{h}$ is a strong subdigraph of $D$. Analogous to the proof of Case 1, we  divide the discussion into the subcases as follows:

		\vspace{0.1cm}
		
		\textbf{(1). $k_{h}\geq2$}. Similar to the previous proof, one easily sees that $G_{h}\in D_{n_{h}+1}^{3}(2,1)$. As $G_{h}$ is a $\emph{$C_{\leq 3}$-free}$ strong subdigraph of $D$, Lemma~\ref{n/2} gives us that $n_{h}\geq 6$.
		
		\vspace{0.1cm}
		\textbf{(1.1). $n_{h}=6$}. Note that $G_{h}$ is isomorphic to the circulant digraph $C_{7}(1,2)$ by Lemma~\ref{n=7}. It is straightforward to check that $\ell_{h}=2$ and $\gamma(G_{h})=7$. Moreover, $\gamma(D_{1},D_{h})\geq \ell_{h}=2$. Together with $\gamma^{*}\geq n-8$, one easily sees that
		\begin{align*}
			\gamma(D)
			&=\gamma(G_{h})+\gamma^{*}+\gamma(D_{1},D_{h})\\
			&\geq 7+n-8+2\\
			&=n+1,
		\end{align*}
		contradicts to the fact that $\gamma(D)=n$.

		\vspace{0.1cm}
		\textbf{(1.2). $n_{h}=7$}. We only need to consider $G_{h}\in \Phi_{8}^{3}(2,1)$ and $G_{h}\notin \Phi_{8}^{3}(2,1)$.
		
		\vspace{0.1cm}
		\textbf{(a)}.  By Lemma~\ref{n=8}, it is obvious that $G_{h}$ is isomorphic to $F_{8}$ if $G_{h}\in \Phi_{8}^{3}(2,1)$. This yields that $\ell_{h}\geq 2$ and $\gamma(G_{h})=8$. Additionally, we have $\gamma(D_{1},D_{h})\geq2$. Combining with $\gamma^{*}\geq n-9$, we can deduce that
		\begin{align*}
			\gamma(D)
			&=\gamma(G_{h})+\gamma^{*}+\gamma(D_{1},D_{h})\\
			&\geq 8+n-9+2\\
			&=n+1,
		\end{align*}
		a contradiction to $\gamma(D)=n$ too.
		
		\vspace{0.1cm}
		\textbf{(b)}. One checks that $\gamma(G_{h})\geq9$ if $G_{h}\notin \Phi_{8}^{3}(2,1)$ by applying Lemma~\ref{n=8}. It is clear that $\gamma(D_{1},D_{h})\geq1$ and $\gamma^{*}\geq n-9$. Consequently, we have
		\begin{align*}
			\gamma(D)
			&=\gamma(G_{h})+\gamma^{*}+\gamma(D_{1},D_{h})\\
			&\geq 9+n-9+1\\
			&=n+1,
		\end{align*}
		a contradiction again.

		\vspace{0.1cm}
		\textbf{(1.3). $n_{h}\geq8$}. Observe first that $9\leq n_{h}+1<n$ and $G_{h}\in D_{n_{h}+1}^{3}(2,1)$. By the induction hypothesis, we obtain that $|A(G_{h})|\leq \varphi_{n_{h}+1}^{3}(2,1)=\binom{n_{h}}{2}-2$. Thus,  $\gamma(G_{h})\geq n_{h}+2$ holds. Recall that $\gamma(D_{1},D_{h})\geq1$ and $\gamma^{*}\geq n-n_{h}-2$. Consequently, we have
		\begin{align*}
			\gamma(D)
			&=\gamma(G_{h})+\gamma^{*}+\gamma(D_{1},D_{h})\\
			&\geq n_{h}+2+n-n_{h}-2+1\\
			&=n+1,
		\end{align*}
		a contradiction.
		
		\vspace{0.1cm}
		
		\textbf{(2). $k_{h}=1$ and $\ell_{h}\leq 2$}
		
		Notice once again that $D_{h}\in D^{3}_{n_{h}}(1,1)$ and $|A(D_{h})|\leq \varphi^{3}_{n_{h}}(1,1)=\binom{n_{h}-1}{2}+1$ holds. Hence, we have $\gamma(D_{h})\geq n_{h}-2$. Analogous to the above discussion in Case 1, $n_{h}\geq 6$ holds. It is not difficult to check that $\gamma(D_{1},D_{h})\geq \ell_{h}$ and $\gamma(w,D_{h})=n_{h}-\ell_{h}-1$. We can derive the following by combining with $\gamma^{*}\geq n-n_{h}-2$:
		\begin{align*}
			\gamma(D)
			&=\gamma(D_{h})+\gamma^{*}+\gamma(D_{1},D_{h})+\gamma(w,D_{h})\\
			&\geq n_{h}-2+n-n_{h}-2+\ell_{h}+n_{h}-\ell_{h}-1\\
			&=n+n_{h}-5\\
			&\geq n+1.
		\end{align*}
		a contradiction.

		\vspace{0.1cm}
		
		\textbf{(3). $k_{h}=1$ and $\ell_{h}\geq 3$}
		
		It is not hard to see that $n_{h}\geq 5$ and $\gamma(D_{1},D_{h})\geq 3$. Moreover, $G_{h}\in D^{3}_{n_{h}+1}(1,1)$ and $n_{h}+1\geq6$. Clearly, $G_{h}$ contains a unique vertex of out-degree one. Using Proposition~\ref{two}, it is straightforward that $G_{h}\notin \Phi^{3}_{n_{h}+1}(1,1)$, implying that $|A(G_{h})|\leq \binom{n_{h}}{2}$ by Corollary~\ref{n=3}. Thus, $\gamma(G_{h})\geq n_{h}$ holds. Due to $\gamma^{*}\geq n-n_{h}-2$, the following holds:
		\begin{align*}
			\gamma(D)
			&=\gamma(G_{h})+\gamma^{*}+\gamma(D_{1},D_{h})\\
			&\geq n_{h}+n-n_{h}-2+3\\
			&=n+1.
		\end{align*}

		This contradicts our assumption that $\gamma(D)=n$.

		\vspace{0.1cm}
		
		\textbf{Subcase 3.2. $D_{1}$ has at least $4$ vertices, i.e., $n_{1}\geq4$}

		Note once again that $D_{1}\in D^{3}_{n_{1}}(1,1)$. Thus, we obtain $|A(D_{1})|\leq \binom{n_{1}-1}{2}+1$, which follows that $\gamma(D_{1})\geq n_{1}-2$. Since $G_{1}$ is not strong, it is clear that $\gamma(w,D_{1})=n_{1}-k_{1}$. In addition, $\gamma(D_{1},D_{h})\geq k_{1}\ell_{h}\geq k_{1}$ and $\gamma^{*}\geq n-n_{1}-n_{h}-1$. Moreover, by the previous discussion, one easily checks that $\gamma(G_{h})\geq n_{h}$. By combining these conclusions, one can deduce that
		\begin{align*}
			\gamma(D)
			&=\gamma(D_{1})+\gamma(G_{h})+\gamma^{*}+\gamma(D_{1},D_{h})+\gamma(w,D_{1})\\
			&\geq n_{1}-2+n_{h}+n-n_{1}-n_{h}-1+k_{1}+n_{1}-k_{1}\\
			&=n+n_{1}-3\\
			&\geq n+1,
		\end{align*}
		contradicts the fact that $\gamma(D)=n$.

		\vspace{0.1cm}
		
		\noindent\textbf{Case 4. Both $G_{1}$ and $G_{h}$ are not strong}
		
		\vspace{0.1cm}

		Note once again that $n_{1}=1$ or $n_{1}\geq4$, and $n_{h}\geq4$. In the remaining proof of this case, we divide the discussion into the following two subcases:

		\vspace{0.1cm}
		
		\textbf{Subcase 4.1. $D_{1}$ contains only one vertex, i.e., $n_{1}=1$}

		One easily sees that $k_{1}=1$ as $D$ is strong. Since $G_{h}$ is not strong, then $k_{h}=0$. Recall that $X=N^{+}_{D}(w)\cap V(U)$ and $|X|=x$. It is obvious that $x\geq 1$ due to $d^{+}_{D}(w)\geq 2$. It is not difficult to check that $\gamma^{*}\geq n-n_{h}+\ell_{h}-3$ and $\gamma(D_{1},D_{h})\geq \ell_{h}$. Together with $\gamma(D_{h})\geq n_{h}-2$, the following holds:
		\begin{align*}
			\gamma(D)
			&=\gamma(D_{h})+\gamma^{*}+\gamma(D_{1},D_{h})+\gamma(w,D_{h})\\
			&\geq n_{h}-2+n-n_{h}+\ell_{h}-3+\ell_{h}+n_{h}-\ell_{h}\\
			&=n+\ell_{h}+n_{h}-5.
		\end{align*}
		
		For the same reason presented in the proof of Case 2, it is not hard to check that $\ell_{h}+n_{h}\geq7$. Consequently, we have $\gamma(D)\geq n+2$, contradicting the fact that $\gamma(D)=n$.

		\vspace{0.1cm}
		
		\textbf{Subcase 4.2. $D_{1}$ contains at least $4$ vertices, i.e., $n_{1}\geq4$}
		
		Obviously, $k_{1}$ and $\ell_{h}$ are positive integers and $\ell_{1}=k_{h}=0$. Notice once again that $n_{i}\geq 4$ and $D_{i}\in D^{3}_{n_{i}}(1,1)$ for $i\in\{1,h\}$. Hence, we have $|A(D_{i})|\leq \binom{n_{i}-1}{2}+1$, yielding that $\gamma(D_{i})\geq n_{i}-2$, where $i=1,h$. Because $G_{1}$\;(resp., $G_{h}$) is not strong, one can check that $\gamma(w,D_{1})=n_{1}-k_{1}$\;(resp., $\gamma(w,D_{h})=n_{h}-\ell_{h}$). Additionally, $\gamma(D_{1},D_{h})\geq k_{1}\ell_{h}$ and $\gamma^{*}\geq n-n_{1}-n_{h}-1$. Combining  these conclusions, one can deduce that
		\begin{align*}
			\gamma(D)
			&=\gamma(D_{1})+\gamma(D_{h})+\gamma^{*}+\gamma(D_{1},D_{h})+\gamma(w,D_{1})+\gamma(w,D_{1})\\
			&\geq n_{1}-2+n_{h}-2+n-n_{1}-n_{h}-1+k_{1}\ell_{h}+n_{1}-k_{1}+n_{h}-\ell_{h}\\
			&=n+n_{1}+n_{h}-6+(k_{1}-1)(\ell_{h}-1)\\
			&\geq n+n_{1}+n_{h}-6\\
			&\geq n+2,
		\end{align*}
		a contradiction to $\gamma(D)=n$.          \hfill $\blacksquare$

		\vspace{0.2cm}
		
		Combining the previous proofs, we thus complete the proof of Theorem~\ref{10}.

		\section{Concluding remarks}
		\noindent
		
		Recall that a digraph is \emph{$C_{\leq k}$-free} if it contains no cycle of length at most $k$. For positive integers $\xi$ and $\zeta$, $D_{n}^{k}(\xi,\zeta)$ is a family of \emph{$C_{\leq k}$-free} strong digraphs on $n$ vertices such that each digraph $D\in D_{n}^{k}(\xi,\zeta)$ satisfies $\delta^{+}(D)\geq\xi$ and $\delta^{-}(D)\geq\zeta$. Let $\varphi_{n}^{k}(\xi,\zeta)=\max\{|A(D)|:D\in D_{n}^{k}(\xi,\zeta)\}$. Denote  $\Phi_{n}^{k}(\xi,\zeta)=\{D\in D_{n}^{k}(\xi,\zeta):|A(D)|=\varphi_{n}^{k}(\xi,\zeta)\}$. Particularly, $\varphi_{n}^{k}(\xi,\zeta)=0$ if $\Phi_{n}^{k}(\xi,\zeta)=\O$ for some collections of $\{n,k,\xi,\zeta\}$.

		In this paper, we prove that  $\varphi^{3}_{n}(2,1)=\binom{n-1}{2}-2$ for $n\ge 10$. Combining Theorem~\ref{upper-lower bounds} and $\varphi_{n}^{3}(2,1)=0$ for $n\leq 6$, we thus completely determine the values of $\varphi^{3}_{n}(2,1)$ for all $n$. Proposition~\ref{equivalence} establishes the equivalence relation between the problem to determine the exact values of $\varphi^{k}_{n}(\xi,\zeta)$ and the famous Caccetta-H\"{a}ggkvist Conjecture. As a consequence, it will be interesting to investigate ordinary $\varphi^{k}_{n}(\xi,\zeta)$ except $\varphi^{k}_{n}(1,1)$ and $\varphi^{3}_{n}(2,1)$, where $k\geq3$ and $\xi,\zeta$ are positive integers.

		Furthermore, we believe Lemma~\ref{one vertex} may be of independent interest since we observe that if one can find a vertex of a digraph $D\in D_{n}^{k}(\xi,\zeta)$ of small degree such that the removal of this vertex does not destroy the strong connectedness, then the upper bounds of general $\varphi^{k}_{n}(\xi,\zeta)$ might be improved by repeatedly applying this lemma.

		\subsection*{Acknowledgments}

		\noindent
		The work was supported by the National Natural Science Foundation of China (No. 12071453), the National Key R and D Program of China(2020YFA0713100),  the Anhui Initiative in Quantum Information Technologies (AHY150200)  and the Innovation Program for Quantum Science and Technology, China (2021ZD0302904).

		\vskip 3mm
	\end{spacing}
\end{document}